\documentclass[12pt]{amsart}

\usepackage{amscd,amssymb,amsthm,verbatim}
\usepackage{enumerate}
\usepackage{bm}
\usepackage[margin=1in]{geometry}
\usepackage{mathrsfs, graphicx} 
\usepackage{stmaryrd}

\usepackage[%
bookmarks=true,
colorlinks,
linkcolor=blue,
urlcolor=blue,
citecolor=blue,
plainpages=false,
pdfpagelabels,
final,
breaklinks=true\between
]{hyperref}
\usepackage{hyperref}
\usepackage[numbers,sort&compress]{natbib}

\numberwithin{equation}{section}

\theoremstyle{plain}
\newtheorem{thm}{Theorem}[section]
\newtheorem{prop}[thm]{Proposition}
\newtheorem{lem}[thm]{Lemma}
\newtheorem{cor}[thm]{Corollary}
\theoremstyle{remark}\newtheorem{claim}[thm]{Claim}

\theoremstyle{definition}
\newtheorem{defn}[thm]{Definition}
\theoremstyle{remark}
\newtheorem{rk}[thm]{Remark}
\newcommand{\ve}{\varepsilon}
\newcommand{\les}{\lesssim}

\newcommand{\ind}{\operatorname{ind}}

\newcommand{\tr}{\operatorname{tr}}

\newcommand{\Ric}{\mathrm{Ric}}
\newcommand{\Riem}{\mathrm{Riem}}

\newcommand{\loc}{\mathrm{loc}}

\newcommand{\Div}{\operatorname{div}}

\DeclareFontFamily{U}{MnSymbolC}{}
\DeclareSymbolFont{MnSyC}{U}{MnSymbolC}{m}{n}
\DeclareFontShape{U}{MnSymbolC}{m}{n}{
	<-6>  MnSymbolC5
	<6-7>  MnSymbolC6
	<7-8>  MnSymbolC7
	<8-9>  MnSymbolC8
	<9-10> MnSymbolC9
	<10-12> MnSymbolC10
	<12->   MnSymbolC12}{}
\DeclareMathSymbol{\intprod}{\mathbin}{MnSyC}{'270}

\begin{document}
\title[Density and mass for data sets with boundary]
{Density and positive mass theorems for initial data sets with boundary}
\author{Dan A.\ Lee, Martin Lesourd, and Ryan Unger}

\address{Graduate Center and Queens College, City University of New York, 365 Fifth Avenue, New York, NY 10016}
\email{dan.lee@qc.cuny.edu}

\address{Black Hole Initiative, Harvard University, Cambridge, MA 02138}
\email{mlesourd@fas.harvard.edu}

\address{Department of Mathematics, Princeton University, Princeton, NJ 08544}
\email{runger@math.princeton.edu}

\maketitle

\begin{abstract}
We prove a harmonic asymptotics density theorem for asymptotically flat initial data sets with compact boundary that satisfy the dominant energy condition. We use this to settle the spacetime positive mass theorem, with rigidity, for initial data sets with apparent horizon boundary in dimensions less than $8$ without a spin assumption. 
\end{abstract}

\tableofcontents

\section{Introduction}
 
In the study of asymptotically flat manifolds, it is often useful to perturb the geometric data into something easier to work with. 
The first result of this kind was proved by R.~Schoen and S.-T.~Yau \cite{SY81ii}, who showed how to perturb an asymptotically flat, scalar-flat metric 
\begin{equation}\label{a}g_{ij}(x)=\delta_{ij}+ O_{2}(|x|^{-1}), \end{equation}
to an asymptotically flat, scalar-flat metric  satisfying 
\begin{equation}\label{b}\tilde g_{ij}(x)=u^4(x)\delta_{ij}\end{equation}
outside a compact set, where $u$ is a harmonic function with respect to the underlying Euclidean structure and $u(x)\to 1$ as $|x|\to\infty$. Such a metric has so-called Schwarzschild asymptotics, which can be seen by expanding $u$ in spherical harmonics, and so may be treated using the methods of \cite{SY79PMT}. 
This result can be generalized to show that an asymptotically flat, nonnegative scalar curvature metric can also be perturbed so that~\eqref{b} holds outside a compact set, while maintaining nonnegative scalar curvature \cite{LeeParker, Kuwert}. One may characterize these results as saying that metrics with asymptotic expansion \eqref{b} are \emph{dense} among metrics with asymptotic expansion \eqref{a}, either subject to the constraint $R_g=0$, or the constraint $R_g\ge0$, where $R_g$ denotes the scalar curvature of $g$.

Density theorems of this kind are used in every minimal hypersurface, marginally outer trapped surface (MOTS), and Jang reduction proof of the positive mass theorem \cite{SY79PMT, SY81ii, SY81i, SchoenPMT, L99, E13, EHLS, L16, L17, SY17}. The improved asymptotics are used to construct barriers and perform asymptotic analysis on area-minimizing hypersurfaces (or stable MOTS) spanning large spheres in the asymptotically flat region. We note, however, that the conformally flat structure in \eqref{b} is not necessary for the analysis; one merely needs that the metric is asymptotic to a Schwarzschild metric up to the mass order term \cite{CarlottoThesis, LUY21}. J.~Lohkamp observed in \cite{L99} that the density theorem described above can be used to compactify an asymptotically flat metric with negative mass, thereby reducing the Riemannian positive mass theorem to the theorem that a connected sum of a torus with a compact manifold cannot carry a metric with positive scalar curvature \cite{SY79, GL80}. See also \cite{CP11, L16, SY17, LUY20} for simplifications and further applications of this technique. 

Recall that $(M, g, k)$ is an \emph{initial data set} if $(M, g)$ is a Riemannian manifold equipped with a symmetric $(0,2)$-tensor field $k$, and that the \emph{mass density} $\mu$ (a scalar) and the \emph{momentum density} $J$ (a vector field) are defined by
\begin{align*}
    \mu &=\tfrac 12\left(R_g-|k|_g^2+(\tr_gk)^2\right)\\
   J^i &= (\Div_g k)^i -\nabla^i (\tr_g k).
\end{align*}
We say that $(M,g,k)$ satisfies the \emph{dominant energy condition} (\emph{DEC}) if $\mu \ge |J|_g$. 
It is often convenient to replace $k_{ij}$ by the \emph{momentum tensor}
\[\pi^{ij}=k^{ij}-(\tr_gk)g^{ij},\]
and we will (abusively) also refer to $(M, g, \pi)$ as an initial data set.

The first density theorem for initial data sets was proved by J.~Corvino and Schoen \cite{CorvinoSchoen},
who showed that asymptotically flat vacuum initial data sets $(g,\pi)$ can be approximated by vacuum initial data satisfying the \emph{harmonic asymptotics} condition  
\begin{equation}
    \tilde g_{ij}=u^4\delta_{ij},\qquad \tilde\pi^{ij}= u^{-6}\mathfrak L_\delta Y^{ij}
\end{equation}
outside a compact set, for some function $u$ and vector field $Y$ which have good asymptotic expansions. The notation $\mathfrak L$ is defined by the formula $\mathfrak L_g Y^{ij}:= (L_Y g)^{ij} - (\Div_g Y)g^{ij}$ for an arbitrary metric $g$, where $L_Y$ denotes the Lie derivative.
In the extension of the positive mass theorem to initial data sets in higher dimensions \cite{EHLS}, M.~Eichmair, Lan-Hsuan Huang, 
the first named author, and Schoen proved a harmonic asymptotics density theorem for initial data sets satisfying the dominant energy condition. This theorem also plays a crucial role in Eichmair's Jang reduction and rigidity theorems \cite{E13}, as well as Lohkamp's compactification approach to the spacetime positive mass theorem \cite{L17}. 

In the present paper, we generalize the density theorems of Corvino--Schoen~\cite{CorvinoSchoen} and Eichmair--Huang--Lee--Schoen~
\cite{EHLS} to allow for initial data sets with compact boundary. 

\begin{thm}[Density theorem for initial data sets with boundary] \label{DECdensity}
Let $(M^n,g,k)$ be a complete asymptotically flat initial data set with compact boundary $\partial M$, such that the dominant energy condition, $\mu \ge |J|_g$, holds on $M$. Let $p>n$ and $\frac{n-2}{2}<q<n-2$ such that $q$ is less than the decay rate of $(g,k)$. Let $\theta^+$ denote the outer null expansion of $\partial M$.

Then for any $\ve>0$, there exists an asymptotically flat initial data set $(\tilde g,\tilde k)$ on $M$ also satsifying the dominant energy condition such that $(\tilde g,\tilde k)$ has harmonic asymptotics in each end of $M$, $(\tilde g,\tilde k)$ is $\ve$-close to $(g,k)$ in $W^{2,p}_{-q}\times W^{1,p}_{-q-1}$, the new constraints $(\tilde \mu,\tilde J)$ are $\ve$-close to $(\mu,J)$ in $L^1$, and the new outer null expansion $\tilde \theta^+$ is \emph{equal to} $\theta^+$ on $\partial M$. 

Furthermore, we can choose $(\tilde g,\tilde k)$ such that the \emph{strict} dominant energy condition holds, $\tilde \mu>|\tilde J|_g$. Simultaneously, $(\tilde \mu,\tilde J)$ may be chosen to decay as fast as we like. 

Alternatively, we can choose $(\tilde g,\tilde k)$ to be vacuum (that is, $\tilde\mu =|\tilde J|_g=0$) outside a compact set. Moreover, if $(g,k)$ is vacuum everywhere, then $(\tilde g,\tilde k)$ can be chosen to be vacuum everywhere. 
\end{thm}

\begin{rk}
More generally, we may prescribe $\tilde \theta^+$ to be any function sufficiently close to $\theta^+$ in the fractional Sobolev space $W^{1-\frac 1p,p}(\partial M)$. This theorem is more precisely stated as Theorem \ref{DECdensity2} below. In particular, $\tilde \theta^+$ may be chosen to be strictly less than $\theta^+$ at every point. Moreover, we may alternatively choose to prescribe the inner expansion $\theta^-$ (instead of $\theta^+$) on any given components of $\partial M$.
\end{rk}

The perturbation described in Theorem~\ref{DECdensity} only changes the ADM energy-momentum by a small amount, and hence we can use our density theorem to settle the full spacetime positive mass theorem with boundary and without a spin assumption, at least in dimensions where we have regularity of $C$-minimizing integral currents. 

\begin{thm}[Spacetime positive mass theorem with boundary] \label{PMT}
Let $3\le n\le 7$, and let $(M^n,g,k)$  be a complete asymptotically flat initial data set with compact boundary $\partial M$ such that the dominant energy condition holds on $M$, and each component of $\partial M$ is either weakly outer trapped ($\theta^+\le 0$) or weakly inner untrapped ($\theta^-\ge 0$), with respect to the normal pointing into $M$. Then $E\ge |P|$ in each end, where $(E,P)$ denotes the ADM energy-momentum vector of $(g,k)$. 
\end{thm}

The $\partial M=\emptyset$ case of this theorem was proved by Eichmair, Huang, the first author, and Schoen in~\cite{EHLS}.
The $\partial M\neq\emptyset$ case is desirable from a physical perspective, since it verifies the intuition that the geometry lying behind an ``apparent horizon'' cannot influence the asymptotic geometry. Previously, the $\partial M\neq\emptyset$ case  was only known to be true for spin manifolds, by work of M. Herzlich \cite{Herzlich} (see \cite{GHHP}), who implemented Witten's method \cite{Witten} with a boundary condition. In 3 dimensions, Theorem~\ref{PMT} also follows from recent work of S.~Hirsch, D.~Kazaras, and M.~Khuri~\cite{HKK}, using an unrelated method. The recent note of G.~Galloway and the first author \cite{GallowayLee} proves Theorem~\ref{PMT} under the stronger assumption that each component of $\partial M$ either has $\theta^+< 0$ or $\theta^-> 0$.

Despite these advances and the general belief that Theorem~\ref{PMT} is true, the problem has remained open until now. It is natural to adapt the proof of the $\partial M=\emptyset$ case in~\cite{EHLS}, and in fact, the proof is essentially unchanged for $\partial M\neq \emptyset$ \emph{if} one already has harmonic asymptotics. However, it is not clear how to achieve harmonic asymptotics as in~\cite[Theorem 18]{EHLS} when a boundary is present. This is what we accomplish with Theorem~\ref{DECdensity}, and we explain how Theorem~\ref{PMT} follows from Theorem~\ref{DECdensity} and~\cite{EHLS} in Section \ref{S4.inequality}.
The reason why Theorem~\ref{DECdensity} is a nontrivial generalization of Theorem 18 of~\cite{EHLS} is that the latter is proved by solving an elliptic system, and the weakly outer trapped condition on the boundary is not an elliptic boundary condition for this system. We solve this problem by supplementing the weakly outer trapped condition with other conditions to create an elliptic boundary condition.

Since one expects that $E=|P|$ is only possible if the initial data set sits inside Minkowski space, which does not contain weakly outer trapped surfaces, one should be able to strengthen Theorem~\ref{PMT} to conclude that $E>|P|$. Indeed, we are able to do this if one is willing to make stronger assumptions about the asymptotics.

\begin{thm}\label{strictPMT}
Assume the hypotheses of Theorem~\ref{PMT} with $\partial M\neq \emptyset$, and furthermore, assume that $(M, g, k)$ satisfies the stronger asymptotic assumption appearing in Theorem~\ref{ADMnull}. Then $E>|P|$.
\end{thm}

 By work of R.~Beig and P.~Chru\'{s}ciel~\cite{BeigChrusciel} (see also~\cite{ChruscielMaerten}), this result should also hold for all spin manifolds, and this argument is sketched\footnote{ Note that this discussion only appears in the arXiv version of the paper.} in \cite[Remark 11.5]{BartnikChrusciel} in dimension 3. The 3 dimensional case was also obtained by~\cite{HKK}, and more recently, Hirsch and Yiyue Zhang used this approach to remove the ``stronger asymptotic assumption'' in 3 dimensions~\cite{HirschZhang}. Note that the Riemannian case of Theorem~\ref{strictPMT} is a direct consequence of the Riemannian Penrose inequality, which is known in dimensions $\le7$~\cite{BrayLee}.

Our proof of Theorem~\ref{strictPMT} follows fairly easily from Theorem~\ref{PMT} combined with known techniques in the $\partial M=\emptyset$ case. Specifically, we break the proof into two parts. In the first part, presented in Section \ref{S4.1}, we suppose that $E=|P|$ and conclude that $E=|P|=0$ by adapting the $\partial M=\emptyset$ proof by Huang and the first author~\cite{HuangLee}. This is where the stronger asymptotic assumption in Theorem~\ref{ADMnull} is needed. In the second part, presented in Section \ref{S4.2}, we show that $E=0$ is impossible by examining Eichmair's Jang equation proof (in the $\partial M=\emptyset$ case) that $E\ge0$ in~\cite{E13} (which itself generalized Schoen--Yau's pioneering result in dimension 3~\cite{SY81ii}). Technically, in dimension 3 our argument requires the assumption that $\tr_g k=O(|x|^{-\gamma})$ for some $\gamma>2$, but we choose to leave this assumption out of the statement of Theorem~\ref{strictPMT} by explicitly relying on either~\cite{HKK} or~\cite[Remark 11.5]{BartnikChrusciel} (both of which require very different methods than the ones discussed in this article). 

Besides the positive mass theorem, 
another application of Theorem \ref{DECdensity} concerns the gluing problem for initial data sets. Indeed, since the gluing-across-annulus theorem of Corvino--Schoen \cite{CorvinoSchoen} is appropriately local and done in a region where the data is vacuum and has good asymptotics, we can combine it with Theorem \ref{DECdensity} to obtain the following.
\begin{cor}\label{gluingproblem}
Let $(M,g,k)$ be an asymptotically flat initial data set satisfying the assumptions of Theorem \ref{DECdensity}, such that $E>|P|$. Then, for any $\varepsilon>0$, there is an initial data set $(\tilde g,\tilde k)$ with the following properties:
\begin{itemize}
    \item $(\tilde g,\tilde k)$ satisfies the dominant energy condition, 
    \item the outer null expansion of $\partial M$ with respect to $(\tilde g,\tilde k)$ is unchanged, that is,  $\tilde \theta^+=\theta^+$,
    \item $(\tilde g,\tilde k)$ is $\varepsilon$-close to $(g, k)$ in $W^{2,p}_{-q}\times W^{1,p}_{-q-1}$,
    \item $(\tilde\mu, \tilde J)$ is  $\varepsilon$-close to $(\mu, J)$ in $L^1$, 
    \item outside a compact set containing $\partial M$, $(\tilde g,\tilde k )$ is isometric to an initial data set for a Kerr spacetime\footnote{More specifically, this Kerr initial data comes from an element of the ``reference family'' for Kerr, as described in~\cite{ChruscielDelay}.} with ADM energy-momentum $(\tilde E, \tilde P)$, where $|\tilde E- E| + |\tilde P- P| < \varepsilon$.
\end{itemize}
\end{cor}

Finally, we note that a version of Corollary \ref{gluingproblem} exists for the gluing result of Carlotto--Schoen \cite{CarlottoSchoen}. The gluing there is done along the boundary of a cone which goes out to infinity, and the geometry at the gluing interface has to satisfy a certain smallness which is implemented by taking the cone to lie sufficiently far out in the asymptotic region. Given $N$ asymptotically flat initial data sets each satisfying the assumptions of Theorem \ref{DECdensity}, we can use Theorem \ref{DECdensity} to perturb each of these initial data sets to $N$ new initial data sets which are all vacuum outside a compact set and satisfy harmonic asymptotics. One can then glue these $N$ new initial data sets together using \cite{CarlottoSchoen} to produce a final asymptotically flat initial data set with the following properties: 
\begin{itemize}
\item the dominant energy condition holds,
\item there are $N$ boundary components, each of which has the same outer null expansion as the original initial data sets,
\item the initial data set is vacuum outside a compact set, 
\item the geometry is supported in $N$ (non-intersecting) cones, that is, the region between these cones is Euclidean, with vanishing $k$.
\end{itemize}

If one starts with $N$ initial data sets with outermost MOTS (marginally outer trapped surface) boundaries, one might hope to think of the object resulting from this construction as a model for a black hole with $N$ components. However, it is not clear whether the MOTS boundary of this new initial data set will be \emph{outermost}. It is an interesting  question whether one can guarantee the outermost property, under some reasonable assumptions. \\ \indent 
On this point, we note that P.~Chru\'sciel and R.~Mazzeo \cite{ChruscielMazzeo} have constructed initial data sets whose apparent horizon is composed of multiple connected components. For a complementary perspective in the setting of positive cosmological constant $\Lambda>0$, P.~Hintz \cite{Hintz} constructed a spacetime modelled on Schwarzschild--de Sitter whose future conformal boundary defines an event horizon with multiple connected components.

\vspace{3mm}

\textbf{Acknowledgments.} We thank Lan-Hsuan Huang for useful  discussions at the start of this project, Greg Galloway for his interest in the problem, and Piotr Chru\'sciel for various helpful comments. 

\section{Preliminaries} 

\subsection{Notation and definitions} \label{defs}

Let $M^n$ be a smooth $n$-dimensional manifold ($n\ge 3$) with compact boundary $\partial M$, and fix a smooth background metric $\overline g$ which is identically Euclidean on the finitely many noncompact ends of $M$, which are all diffeomorphic to $\mathbb R^n$ minus a ball. In this setting there are natural definitions of weighted Sobolev spaces $W^{k,p}_{s}$ and weighted H\"{o}lder spaces $C^{k,\alpha}_{s}$, as in \cite{Lee}, for instance.

\begin{defn}\label{AFdata}
We say that an initial data set $(M^n,g,k)$ is \emph{asymptotically flat} if $(g,k)$ is locally $C^{2,\alpha}\times C^{1,\alpha}$ for some $0<\alpha<1$, and there exists a compact set $K\subset  M$ such that $M\setminus K$ is a disjoint union of Euclidean ends such that in the associated coordinate charts,
\begin{align}
    g_{ij}(x)&=\delta_{ij}+O_2(|x|^{-q})\\
    k_{ij}&= O_1(|x|^{-q-1})
\end{align}
for some $q>\frac{n-2}{2}$, and also $(\mu,J)\in L^1(M)$. We refer to this $q$ as the asymptotic decay rate of $(g,k)$. 

In this case, the ADM energy-momentum $(E,P)$ is well-defined. We refer the reader to \cite{Lee} for details and references. 
\end{defn}

Throughout most of this paper, we will use $\pi$ in place of $k$, as described in the introduction. Given a fixed manifold $M$, we define the \emph{constraint map} $\Phi$ by 
\[ \Phi(g,\pi)=(2\mu ,J),\]
for any initial data $(g,\pi)$ on $M$.

Given a hypersurface $\Sigma$ with unit normal $\nu$ in an initial data set $(M, g, k)$, we define the \emph{outer and inner null expansions} $\theta^+_\Sigma$ and $\theta^-_\Sigma$, respectively, with respect to $(g,\pi)$ by 
 \[\theta^\pm_\Sigma=\pm H_\Sigma+\tr_{\Sigma} k,\]
 where $H_\Sigma$ is the mean curvature of $\partial M$ with respect to $g$ and $\nu$, and 
 \[\tr_{\Sigma} k=(g^{ij}-\nu^i\nu^j)k_{ij}=-\pi^{ij}\nu_i\nu_j\]
 is the trace of $k$ over $T\Sigma$. In the case when $(M,g,k)$ sits inside a spacetime, $\theta_\Sigma^\pm$ can be interpreted in terms of Lorentzian geometry, but we shall not need this viewpoint here. In this paper we will always choose $\Sigma$ to be $\partial M$, and we choose $\nu$ to be the unit normal pointing into $M$. We will want to prescribe either  $\theta_{\partial M}^+$ or $\theta_{\partial M}^-$ on each boundary component, so we make the following definition.

 \begin{defn}\label{boundary}
 Let $M$ be a fixed manifold with boundary, and let $\partial^+ M$ and  $\partial^- M$ designate unions of components of $\partial M$ such that $\partial M=\partial^+ M \cup \partial^- M$.
 Given initial data $(g,\pi)$ on $M$, define $\Theta(g,\pi)$ to be the function $\partial M$ that is equal to $\theta_{\partial M}^\pm$ on $\partial^\pm M$ with respect to the data $(g,\pi)$ and the normal pointing into $M$.
\end{defn}

For PDE purposes, it is convenient to slightly enlarge the space of data sets under consideration. We will consider initial data $(g,\pi)$, where $g- \overline{g}\in W^{2,p}_{-q}(T^*M \odot T^* M)$ and $\pi\in W^{1,p}_{-q-1}(TM\odot TM)$, where $p>n$, $\frac{n-2}{2}<q<n-2$, and $q$ is less than the decay rate in Definition \ref{AFdata}. Note that such a pair $(g,\pi)$ need not satisfy our definition of asymptotic flatness, and in particular, need not have well-defined ADM energy-momentum. We define 
\begin{equation}\label{eq:definesD}
    \mathcal D:=\left(\overline{g} + W^{2,p}_{-q}(T^*M \odot T^* M) \right)\times W^{1,p}_{-q-1}(TM\odot TM),
\end{equation}
so that $\mathcal D$ is a (affine) Banach space of initial data sets. Note that the tangent space of $\mathcal D$ at $(g,\pi)$, $T_{(g,\pi)}\mathcal D$, can be identified with $W^{2,p}_{-q}(T^*M\odot T^*M)\times W^{1,p}_{-q-1}(TM\odot TM)$.

 \begin{lem}\label{theta_regularity}
Let  $p>n$ and $\frac{n-2}{2}<q<n-2$. On a fixed asymptotically flat manifold $M^n$ with compact boundary decomposed as in Definition \ref{boundary}, the descriptions of $\Phi$ and $\Theta$ given above define a smooth map of Banach spaces
 \[(\Phi, \Theta):\mathcal D\to \mathcal L\times W^{1-\frac 1p,p}(\partial M),\]
 where 
 \begin{equation}\label{eq:definesL}
      \mathcal L:=L^p_{-q-2}(M)\times L^p_{-q-2}(TM),
  \end{equation}
and $W^{1-\frac 1p,p}(\partial M)$ is a fractional Sobolev space on $\partial M$.  (See, for example, \cite[Section 1.4]{Grisvard}.)
 \end{lem}

  \begin{proof}
The claim about $\Phi$ is standard and easy to verify, so we focus on the map $\Theta$. We can realize $\partial M$ as a level set of a smooth function $f$, which has no critical points in a small neighborhood $U$ of $\partial M$.  Then the formula 
 \[\nu^i:=g^{ij}\frac{\partial_i f}{|\nabla f|_g}\]
 defines a vector field on $U$ which is the unit normal to the level sets of $f$ in $U$, and it is has $W^{2,p}$ regularity.
Next, the formula
 \[H_g:=(g^{ij}-\nu^j\nu^j)\frac{1}{|\nabla f|_g}(\partial_{ij}f-\Gamma^k_{ij}\partial_k f),\]
 defines a function on $U$ which is equal to the mean curvature of the level sets of $f$ in $U$. We can also see that $H_g$ has $W^{1,p}$ regularity since $\Gamma^k_{ij}\in W^{1,p}$ and $W^{1,p}$ is a Banach algebra. Similary, the quantity $-\pi_{ij}\nu^i\nu^j$ is a $W^{1,p}$ function on $U$. More precisely, we observe that we have a bounded map from
 from $(g,\pi)\in \mathcal D$ to $\pm H_g -\pi_{ij}\nu^i\nu^j \in W^{1,p}(U)$. The result follows from viewing $\Theta$ as the composition of this map with the usual bounded trace operator from $W^{1,p}(U)$ to $W^{1-\frac 1p,p}(\partial M)$.
 \end{proof}
 
 The proof above made use of a trace theorem. Later on, we will need the following \emph{sharp trace theorem}.
 
 \begin{lem}\label{trace}
 Let $(M^n,\overline g)$ be as above and let $g\in \overline g+ W^{2,p}_{-q}(T^*M\odot T^*M)$. Then, for any $s\in\mathbb R$, the weighted Sobolev space $W^{2,p}_s(M)$ enjoys a bounded trace operator 
 \[T_2^g:W^{2,p}_s(M)\to W^{2-\frac 1p,p}(\partial M)\times W^{1-\frac 1p,p}(\partial M)\]
 which is the unique extension of 
 \[u\longmapsto \left(u|_{\partial M},\left.\frac{\partial u}{\partial \nu_g}\right|_{\partial M}\right)\]
 for $u\in C^2( M)\cap W^{2,p}_s(M)$. The mapping $T_2^g$ is surjective and admits a bounded right inverse. The operator norms of $T_2^g$ and its right inverse depend only on $\|g-\overline g\|_{W^{2,p}_{-q}}$. 
 \end{lem}
 
 We emphasize that the normal vector field $\nu_g$ is the one corresponding to the metric $g$. 
 
 \begin{proof}
The existence and properties of $T_2^{g}$ are easily reduced to the case of bounded domains \cite[Theorem 1.5.1.2]{Grisvard} by means of cutoff functions. In particular, we may take elements in the image of the right inverse to be supported in a neighborhood of $\partial M$.
 \end{proof}

 \subsection{``Conformal" initial data sets} Conformal transformations play a special role in the study of mass and the Riemannian positive mass theorem. The following notion of conformal transformations of initial data sets plays a crucial role in the density theorem and the positive mass theorem \cite{CorvinoSchoen, EHLS}. 
  
 Let 
 \begin{equation}\label{eq:definesC}
 \mathcal C=\left(1+W^{2,p}_{-q}(M)\right)\times W^{2,p}_{-q}(TM)
 \end{equation}
 denote the (affine) Banach  space of conformal deformations. Note that the tangent space of $\mathcal C$ at $(1,0)$, $T_{(1,0)}\mathcal C$, can be identified with $W^{2,p}_{-q}(M)\times W^{2,p}_{-q}(TM)$. For $(g,\pi)\in \mathcal D$ fixed and $(u,Y)\in\mathcal C$, we define 
 \begin{align*}
\tilde g&=u^s g\\
\tilde \pi &= u^{-\frac 32 s}(\pi+\mathfrak L_g Y),
 \end{align*}
 where $s=\frac{4}{n-2}$ is the conformal exponent and $\mathfrak L_g Y$ was defined in the introduction.
 We denote
 \begin{align*}
 \Psi_{(g,\pi)}:\mathcal C&\to \mathcal D \\
  (u,Y)&\mapsto (\tilde g,\tilde\pi).
 \end{align*}

 \begin{defn}\label{harmasymp}
Let $(M^n,g,\pi)$ be an asymptotically flat initial data set. We say that $(M,g,\pi)$ has \emph{harmonic asymptotics} in a particular end if in the asymptotically flat coordinates, the initial data takes the form 
\[(g,\pi)=\Psi_{(\overline g,0)}(u,Y)\]
outside a compact set, where $u$ and $Y$ are a function and vector field pair satisfying 
\begin{align}
    u(x)&=1+a|x|^{2-n}+O_{2,\alpha}(|x|^{1-n})\label{uexp}\\
    Y^i(x)&=b_i|x|^{2-n}+O_{2,\alpha}(|x|^{1-n})\label{Yexp},
\end{align}
for some $\alpha\in(0,1)$. When the initial data set is of this form, the ADM energy-momentum has a particularly simple expression: $E=2a$ and $P_i=-\frac{n-2}{n-1}b_i$. 
\end{defn}
 
 Initial data sets in the image of $\Psi_{(\overline g,0)}$ have harmonic asymptotics if the constraints decay quickly enough:
 
 \begin{lem}[{\cite[Proposition 24]{EHLS}}]\label{EHLS}
 Suppose there exist $(u,Y)\in\mathcal C$ such that $(g,\pi)=\Psi_{(\overline g,0)}(u,Y)$ outside a compact set. If $(\mu,J)\in C^{0,\alpha}_{-n-1-\delta}$ for some $\delta>0$, then $u$ and $Y$ admit the expansions \eqref{uexp} and \eqref{Yexp}. Hence $(g,\pi)$ has harmonic asymptotics. 
 \end{lem}

Next we define
\begin{equation}\label{definesP}
\mathcal P:=(T,\Upsilon):=(\Phi,\Theta)\circ\Psi_{(g,\pi)}:\mathcal C\to \mathcal L\times W^{1-\frac 1p,p}(\partial M).
\end{equation}

In the following, we let $(\mu,J,\theta)$ be the value of $(\Phi,\Theta)$ on the fixed data set $(g,\pi)$. 
 
 \begin{prop}
 The map $\mathcal P$ is a smooth map of Banach spaces and is explicitly given by \begin{align}
     \label{T} T(u,Y)&=\left(u^{-s}\left[\frac{L_gu}{ u}+\tfrac{1}{n-1}(\tr_g\pi+\tr_g\mathfrak L_gY)^2-(|\pi|^2_g+2\langle \mathfrak L_gY,\pi\rangle +|\mathfrak L_gY|_g^2)\right]\right.,\\
     &\quad \left.u^{-\frac 32 s}\left[(\Div_g\mathfrak L_g Y+\Div_g\pi)^i+\tfrac{s(n-1)}{2}(\pi+\mathfrak L_gY)^{ij}\frac{\nabla_ju}{u}-\tfrac s2 \tr_g(\pi+\mathfrak L_gY)g^{ij}\frac{\nabla_j u}{u}\right]\right)\nonumber\\
     \label{Upsilon}\Upsilon(u,Y)&=u^{-\frac s2}\left(\theta \pm \tfrac {s(n-1)}{2}\frac{\partial}{\partial \nu}(\log u)+\Div_{\partial M} Y^\top+H\langle Y,\nu\rangle -\langle \nabla_\nu Y^\perp,\nu\rangle +\langle Y^\top,\nabla_\nu\nu\rangle \right),
 \end{align} 
 where the $\pm$ depends on whether the point lies in $\partial^\pm M$.
 Here $\nu$ is any extension of the $g$-unit normal vector field of $\partial M$ (pointing into $M$), $Y^\perp=\langle Y,\nu\rangle\nu$, and $Y^\top = Y-Y^\perp$. Note that the quantities $\nabla_\nu Y^\perp$ and $\nabla_\nu\nu$ depend on the particular extension chosen, but $-\langle \nabla_\nu Y^\perp,\nu\rangle +\langle Y^\top,\nabla_\nu\nu\rangle $ is an invariant quantity. 
The linearization of $\mathcal P$ at $(1,0)$ is given by 
 \[D\mathcal P_{(1,0)}(v,Z)=(DT|_{(1,0)},D\Upsilon|_{(1,0)})(v,Z),\]
 where 
 \begin{align}
 DT|_{(1,0)}(v,Z)&=\Big(-s(n-1)\Delta_g v+\tfrac{2}{n-1}(\tr_g\pi)(\Div_gZ)-4 \,\pi\cdot\nabla_gZ-2s\mu v,\\
 &\quad\quad\left. (\Div_g \mathfrak L_gZ)^i+\tfrac{s(n-1)}{2} \pi^{ij}\nabla_j v-\tfrac s2(\tr_g\pi)g^{ij}\nabla_jv-\tfrac 32sJ^iv\right),\nonumber\\
D\Upsilon|_{(1,0)}(v,Z)&=-\tfrac s2 \theta v \pm\tfrac{s(n-1)}{2} \frac{\partial v}{\partial \nu}+\Div_{{\partial M}} Z^\top+H\langle Z,\nu\rangle -\langle \nabla_\nu Z^\perp,\nu\rangle +\langle Z^\top,\nabla_\nu\nu\rangle.
 \end{align}
 \end{prop}
 \begin{proof}
 The formula \eqref{T} is given in the erratum for Exercise 9.7 in \cite{Lee}. To prove \eqref{Upsilon}, we first use the standard formula
 \begin{equation}\tilde H=u^{-\frac s2}\left(H+\tfrac {s(n-1)}{2}\frac{\partial}{\partial \nu}(\log u)\right).\end{equation}
 Using $\tilde \nu = u^{-\frac s2} \nu$, we compute 
 \begin{equation}
     \tilde\pi(\tilde \nu,\tilde\nu)= u^{-\frac s2}\left(\pi(\nu,\nu)+\mathfrak L_gY(\nu,\nu)\right).
 \end{equation} 
 Finally, extend $\nu$ off of $\partial M$, let $Y^\perp = \langle\nu,Y\rangle\nu$ and $Y^\top =Y-Y^\perp$. Then, on $\partial M$,
 \begin{align*}
     \mathfrak L_gY(\nu,\nu)&=2\langle\nabla_\nu Y,\nu\rangle -\Div_g Y\\
     &=\langle \nabla_\nu Y,\nu\rangle -\Div_{\partial M} Y\\
     &=-\langle  Y^\top,\nabla_\nu\nu\rangle +\langle\nabla_\nu Y^\perp,\nu\rangle -\Div_{\partial M} Y^\top -H\langle Y,\nu\rangle.
 \end{align*}
Combining these computations gives~\eqref{Upsilon}, and the linearizations are then computed in the obvious way. 
 \end{proof}
 
 Using the formulas for the linearization and the sharp trace theorem, we can prove the following crucial result. 
 
 \begin{lem}\label{boundary_surjective}
The maps \[D\Upsilon|_{(1,0)}:T_{(1,0)}\mathcal C\to W^{1-\frac 1p,p}(\partial M)\] and \[D\Theta|_{(g,\pi)}:T_{(g,\pi)}\mathcal D\to W^{1-\frac 1p,p}(\partial M)\] are surjective and their kernels split. 
\end{lem}
\begin{proof}
 From the formula for $D\Upsilon|_{(1,0)}$, we see that we want to solve the equation 
 \[-\tfrac s2 \theta v \pm\tfrac{s(n-1)}{2} \frac{\partial v}{\partial \nu}+\Div_{{\partial M}} Z^\top+H\langle Z,\nu\rangle -\langle \nabla_\nu Z^\perp,\nu\rangle +\langle Z^\top,\nabla_\nu\nu\rangle=f,\] 
 for any given $f\in W^{1-\frac 1p,p}(\partial M)$. We set $Z\equiv 0$, reducing this to 
  \[-\tfrac s2 \theta v \pm\tfrac{s(n-1)}{2} \frac{\partial v}{\partial \nu}=f \quad\text{on }\partial M.\]
  By the sharp trace theorem  (Lemma \ref{trace}), we can find a $v\in W^{2,p}_{-q}(M)$ such that 
  \[\left( v|_{\partial M}, \left.\frac{\partial v}{\partial \nu} \right|_{\partial M} \right)=\left(0,\frac{\pm 2f}{s(n-1)}\right)\] and 
  \[\|v\|_{W^{2,p}_{-q}(M)}\le C\|f\|_{W^{1-\frac 1p,p}(\partial M)}.\]
  Therefore, $D\Upsilon|_{(1,0)}(v,0)=f$ as desired, with an estimate, which proves $D\Upsilon|_{(1,0)}$ has a bounded right inverse. By standard functional analysis \cite[Theorem 2.12]{Brezis}, this implies that the kernel splits. 
  
  The corresponding statement for $D\Theta|_{(g,\pi)}$ follows by choosing first-order deformations 
 \[ (h,w) = D\Psi_{(g,\pi)}|_{(1,0)}(v,0) = \left( svg, -\tfrac 32 sv\pi\right).\] By the definition of $\Upsilon$ and the chain rule, 
 \[D\Theta|_{(g,\pi)}(h,w)=f\]
 and 
 \[\|(h,w)\|_{T_{(g,\pi)}\mathcal D}\le C \|v\|_{W^{2,p}_{-q}(M)}\le C\|f\|_{W^{1-\frac 1p,p}(\partial M)}.\]
  The same functional analysis argument as before completes the proof. 
\end{proof}

Unfortunately, $D\mathcal P|_{(1,0)}$ does not define an elliptic boundary value problem. This is evident from the fact that the boundary operator $D\Upsilon|_{(1,0)}$ is a scalar operator while $DT|_{(1,0)}$ describes an elliptic system of $n+1$ equations. Therefore, we introduce boundary operators describing $n+1$ equations on the boundary:
 \begin{align*}
     B_1(u,Y)&=u^{-\frac{s}{2}}\left[\theta \pm\tfrac{s(n-1)}{2}\frac{\partial}{\partial\nu}(\log u)\right]\\
     B_2(u,Y)&=Y^\top\\
     B_3(u,Y)&=H\langle Y,\nu\rangle -\langle \nabla_\nu Y^\perp,\nu\rangle +\langle Y^\top,\nabla_\nu\nu\rangle,
 \end{align*}
 where $H$ is the mean curvature of $\partial M$ with respect to $g$. Altogether, these define a map
 \begin{equation}
 B=(B_1,B_2,B_3):\mathcal C\to  W^{1-\frac 1p,p}(\partial M)\times W^{2-\frac 1p,p}(T(\partial M))\times W^{1-\frac 1p,p}(\partial M).
 \end{equation}
 Clearly, we have
 \[\Upsilon=B_1+ u^{-\frac s2}\Div_{\partial M} B_2+u^{-\frac s2}B_3\]
 on $\mathcal C$. It follows that 
 \begin{equation}
 D\Upsilon|_{(1,0)}= DB_1|_{(1,0)}+\Div_{\partial M} DB_2|_{(1,0)}+DB_3|_{(1,0)}
 \end{equation}
on $T_{(1,0)}\mathcal C$, where
\begin{align*}
   DB_1|_{(1,0)}(v,Z) &=-\tfrac s2 \theta v\pm\tfrac {s(n-1)}{2}\frac{\partial v}{\partial \nu}\\
    DB_2|_{(1,0)}(v,Z)&=Z^\top\\
   DB_3|_{(1,0)}(v,Z) &=H\langle Z,\nu\rangle -\langle \nabla_\nu Z^\perp,\nu\rangle +\langle Z^\top,\nabla_\nu\nu\rangle.
\end{align*}
For ease of reading, we will remove $|_{(1,0)}$ when there is no risk for confusion. 

\begin{prop}\label{elliptic}
$(DT|_{(1,0)},DB|_{(1,0)})$ defines an elliptic system in the following sense: 

\begin{enumerate}
 \item \label{EllipticEst}
 There exists a relatively compact set $U\subset M$ so that the elliptic estimate  
    \[\|(v,Z)\|_{W^{2,p}_{-q}}\lesssim \|DT(v,Z)\|_\mathcal L+\|DB(v,Z)\|_{ W^{1-\frac 1p,p}\times W^{2-\frac 1p,p}\times W^{1-\frac 1p,p}}+\|(v,Z)\|_{L^p(U)}\]
    holds for every $(v,Z)\in T_{(1,0)}\mathcal C$, where $\mathcal L$ and $\mathcal C$ were defined in~\eqref{eq:definesL} and~\eqref{eq:definesC}; and

    \item \label{FredholmProperty} the mapping \[(DT,DB):T_{(1,0)}\mathcal C\to \mathcal L\times  W^{1-\frac 1p,p}(\partial M)\times W^{2-\frac 1p,p}(T(\partial M))\times W^{1-\frac 1p,p}(\partial M)\]
    is Fredholm.
    
    \end{enumerate}
\end{prop}

The proof is unfortunately complicated by some technicalities. In order to apply the theory of elliptic systems, we have to check that the boundary operator $DB$ is elliptic. However, we defined $B$ relative to an orthogonal splitting of $Y$ near the boundary. Hence, we must choose coordinates which respect this splitting. The most natural choice would be Fermi coordinates. However, since their construction involves solving the geodesic equation, it is well known that the resulting metric coefficients would only be $C^{0,\alpha}$ \cite{DTK}. This causes problems for the regularity theory, but thankfully exact Fermi coordinates are not needed. Instead, we use the following:

\begin{defn}[Andersson--Chru\'sciel]\label{almostFermi}
Let $g$ be a $C^{k,\alpha}$ Riemannian metric on a manifold $M$ with compact smooth boundary $\partial M$. Let $p\in \partial M$ and $U$ be a neighborhood of $p$ in $ M$. We say that coordinates $(x^1,\dotsc, x^n):U\to \Bbb R^n$ form an \emph{almost-Fermi} coordinate system at $p$ if \begin{enumerate}
\item Each $x^i$ is a $C^{k+1,\alpha}$ function on $U$ relative to the smooth structure of $M$;

\item $x^1,\dotsc,x^{n-1}$ form a coordinate system for a neighborhood of $p$ in $\partial M$ when restricted to $x^n=0$, consequently the coordinate partial derivatives $\partial_1,\dotsc,\partial_{n-1}$ are a frame for $T(\partial M)$ along $\partial M$; and

\item the ``bottom row" of metric components satisfy $g_{nn}(x)=1+O((x^n)^{k+\alpha})$ and $g_{ni}(x)=O((x^n)^{k+\alpha})$ for $i=1,\dotsc, n-1$.
\end{enumerate}
\end{defn}

The existence of almost-Fermi coordinate systems is proved in \cite[Appendix B]{AnderssonChrusciel} (where they are called ``almost Gaussian''). Note that the conclusions of (3) are proved directly there for the inverse metric, but can easily be seen to hold for the metric components themselves by Taylor expansion of the matrix inverse function. We note three more facts: 
\begin{enumerate}
\setcounter{enumi}{3}
\item The metric components are $C^{k,\alpha}$ up to the boundary;

\item the coordinate vector field $\partial_n$ agrees with the $g$-unit normal vector $\nu$ along $\partial M$; and

\item the coordinate vector field $\partial_n$ satisfies $\nabla_{\partial_n}\partial_n=0$ along $\partial M$.
\end{enumerate}

Property (4) follows directly from (1) in Definition \ref{almostFermi}. Property (5) follows from (2) and (3), because $\partial_n$ has unit length on $\partial M$ and is orthogonal to $\partial M$. Property (6) follows from the definition $\nabla_{\partial_n}\partial_n =\Gamma^k_{nn}\partial_k$ and the Christoffel symbols $\Gamma^1_{nn},\dotsc ,\Gamma^n_{nn}$ all vanish along $\partial M$ by (3). 

With this out of the way, we can prove the proposition. 

\begin{proof}[Proof of Proposition \ref{elliptic}]
We compute the boundary operator $DB$ in almost-Fermi coordinates $x^1,\dotsc, x^n$. Let the extension of $\nu$ be $\partial_n$, so that $\nabla_\nu\nu=0$ along $\partial M$.
Then 
\[\langle \nabla_\nu Z^\perp,\nu\rangle = \nabla_\nu\langle Z,\nu\rangle\] along $\partial M$.  We conclude that
\begin{align*}
   DB_1(v,Z) &=-\tfrac s2 \theta v\pm\tfrac{s(n-1)}{2} \frac{\partial v}{\partial x^n},\\
    DB_2(v,Z)&=\sum_{i=1}^{n-1}Z^i\partial_i,\\
   DB_3(v,Z) &= HZ^n- \frac{\partial Z^n}{\partial x^n}.
\end{align*}
We furthermore observe that 
\begin{equation}\label{LaplacianY}
\nabla_j \mathfrak L_g Z^{ij} =\Delta Z^i+\nabla_j\nabla^i Z^j-\nabla^i \nabla_k Z^k= \Delta Z^i+R^i{}_jZ^j,
\end{equation}
so to leading order $DT$ is diagonal and equal to the Laplacian in each component. Moreover, in these coordinates, up to leading order, $DB$ is diagonal and gives a Dirichlet boundary condition in the  $DB_2$ components, while giving a Neumann boundary condition in the $DB_1$ and $DB_3$ components. Therefore it is clear that $(DT,DB)$ is properly elliptic in $M$ and satisfies the complementary condition of Agmon--Douglis--Nirenberg on $\partial M$ \cite{ADNII}. Therefore we have elliptic boundary estimates in addition to interior estimates, which can now be combined with the asymptotic flatness assumption to obtain the global estimate~\eqref{EllipticEst} of Proposition~\ref{elliptic} in routine way. Specifically, we use a partition of unity and a scaling argument to obtain a global weighted estimate (for example, see~\cite[Theorem A.33]{Lee}), and then a cutoff argument to replace $L^p_{-q}(M)$ by $L^p(U)$ on the right-hand side of the global weighted estimate (as in~\cite[Theorem 1.10]{BartnikMass}, or see~\cite[Lemma A.41]{Lee}).

If the coefficients of $(DT, DB)$ were smooth, then the Fredholm property~\eqref{FredholmProperty} of Proposition~\ref{elliptic}  would also follow, as in \cite{LM85}, from the fact that $DB$ is an elliptic boundary condition for $DT$, which is asymptotic to the Laplacian in each component. To account for the lack of smoothness (the coefficients are $C^{1,\alpha}$ at worst), we adapt an argument of D.~Maxwell \cite{Maxwell}. Although the Fredholm property does not follow directly from~\eqref{EllipticEst}, the elliptic estimate~\eqref{EllipticEst} combined with compactness of the map $W^{2,p}_{-q}(M)\to L^p(U)$ does show that the map $(DT, DB)$ is \emph{semi-Fredholm}\footnote{A bounded linear operator $T:X\to Y$ is \emph{semi-Fredholm} if $\dim \ker T<\infty$  and $T(X)$ is closed in $Y$.} via standard arguments~\cite[Theorem 5.21]{Schechter}.

Standard smoothing arguments allow us to construct  a continuous one-parameter family of initial data sets in $(g_\mu, \pi_\mu)\in \mathcal D$ for $\mu\in [0,1]$, such that $(g_\mu,\pi_\mu)\in C^\infty \times C^\infty$ for $\mu>0$, and 
 $(g_0,\pi_0)=(g,\pi)$. Since the associated operators $(DT_\mu,DB_\mu)$ have smooth coefficients for $\mu>0$, we know that they are Fredholm on the relevant Sobolev spaces \cite{LM85}. Since the index of semi-Fredholm operators is a homotopy invariant \cite[Theorem 5.22]{Schechter}, we have $\ind(DT_0,DB_0)=\ind(DT_\mu,DB_\mu)$ for any $\mu>0$. Since the index for $\mu>0$ is finite, this implies that $(DT, DB)=(DT_0,DB_0)$ is itself Fredholm. 
\end{proof}

\section{Density theorems}

\subsection{Prescribed constraint density theorem}

Our first density theorem generalizes the vacuum density theorem of Corvino--Schoen \cite{CorvinoSchoen}. 

\begin{thm}[Density theorem for prescribed constraints]\label{prescribed2}
Let $(M^n,g,\pi)$ be a complete asymptotically flat initial data set with constraints $(\mu,J)$ and compact boundary $\partial M$ having outer null expansion $\theta$ on $\partial^+M$ and inner null expansion $\theta$ on $\partial^-M$. Let $p>n$ and $\frac{n-2}{2}<q<n-2$ be strictly less than the decay rate of $(g,\pi)$. Recall the definitions of $\mathcal L$ and $\mathcal D$ from~\eqref{eq:definesL} and~\eqref{eq:definesD}.
There exist constants $\delta>0$ and $C$ so that the following is true: 

If $( \tilde\mu,\tilde J)\in \mathcal L$, $\tilde\theta \in W^{1-\frac 1p,p}(\partial M)$, and 
$\|(\tilde\mu,\tilde J,\tilde \theta)-(\mu,J,\theta)\|_{\mathcal L\times W^{1-\frac 1p,p}}<\delta,$
then there exists an asymptotically flat initial data set $(\tilde g,\tilde \pi)\in\mathcal D$, whose constraints are $(\tilde \mu,\tilde J)$ and with outer/inner null expansion $\tilde \theta$ on $\partial^\pm M$, which satisfies
\[\|(\tilde g,\tilde \pi)-(g,\pi)\|_{\mathcal D}\le C\|(\tilde\mu,\tilde J,\tilde\theta)-(\mu,J,\theta)\|_{\mathcal L\times W^{1-\frac 1p,p}}.\]
 Furthermore, there exists $(u,Y)\in \mathcal C$ such that 
\[(\tilde g,\tilde \pi)=\Psi_{(\overline g,0)}(u,Y)\]
outside a compact set, where $(\overline g,0)$ is the flat data set on the Euclidean end. 

In particular, if $(\tilde \mu,\tilde J)$ is $C^{0,\alpha}_{-n-1-\ve}$ up to the boundary, and $\tilde\theta\in C^{1,\alpha}(\partial M)$ for some $\alpha\in (0,1)$ and $\ve>0$, then $(\tilde g,\tilde \pi)$ is $C^{2,\alpha}_{2-n}\times C^{1,\alpha}_{1-n}$ up to the boundary and has harmonic asymptotics.\footnote{In this theorem and throughout the paper, whenever we refer to H\"{o}lder spaces on $M$, we mean that they are regular up to the boundary.}
\end{thm}

\begin{rk} 
If $(g,\pi)$ is $C^{k+2,\alpha}\times C^{k+1,\alpha}$ up to the boundary, $(\tilde \mu,\tilde J)$ is $C^{k,\alpha}$ up to the boundary, and $\tilde\theta$ is $C^{k+1,\alpha}$ on the boundary, then $(\tilde g,\tilde \pi)$ will be $C^{k+2,\alpha}\times C^{k+1,\alpha}$ up to the boundary.
\end{rk}

We fix $(g,\pi)$ be asymptotically flat according to Definition \ref{AFdata}, with $(\Phi,\Theta)(g,\pi)= (2\mu,J,\theta)$. We now define the operator used in the proof of Theorem \ref{prescribed2} and study its linearization. 

Let $\chi$ be a smooth nonnegative cutoff function on $\mathbb{R}^n$ equal to $1$ on $B_1$ and vanishing outside $B_2$. Define $\chi_\lambda(x)=\chi(\frac{x}{\lambda})$. For sufficiently large $\lambda$, $\chi_\lambda$ is defined by extending it to be $1$ on the connected compact subset of $M$ that strictly contains the boundary $\partial M$. Now define 
\begin{align*}
    g_\lambda&=\chi_\lambda g+(1-\chi_\lambda)\overline{g}\\
    \pi_\lambda &=\chi_\lambda \pi
\end{align*}
so that $g_\lambda =\overline{g}$ and $\pi_\lambda=0$ for $|x|\ge 2\lambda$ and $\Theta(g_\lambda,\pi_\lambda)=\Theta(g,\pi)=\theta$. It is convenient to set $(g_\infty,\pi_\infty)=(g,\pi)$.

The basic idea of the density theorem (going back to \cite{SY81ii, CorvinoSchoen, EHLS}) is to make a conformal change to $(g_\lambda,\pi_\lambda)$ in order to reimpose the ``constraint" (either prescribed $\Phi$ or modified $\Phi$) lost in the cutoff process by taking $\lambda$ large and using the inverse function theorem. However, $(DT|_{(1,0)},DB_{(1,0)})$ is not necessarily an isomorphism. This issue is also present in \cite{CorvinoSchoen, EHLS}. The solution is to change the domain of $(T,\Upsilon)$ to create an operator whose differential at $(g_\infty,\pi_\infty)$ is an isomorphism. The alteration will only introduce ``compactly supported" deformations, so that the final data set will still have harmonic asymptotics. 

\begin{lem}\label{main} Fix initial data $(g,\pi)$ as in Theorem \ref{prescribed2}. There exists a closed subspace $K_1\subset T_{(1,0)}\mathcal C$, a finite dimensional subspace $K_2\subset T_{(g,\pi)}\mathcal D$ spanned by compactly supported smooth functions, and constants $r_0>0$ and $C$ such that the following holds for all $\lambda$ sufficiently large: 

The differentials of the operators
\begin{align*}
\hat{\mathcal P}_\lambda:[(1,0)+K_1]\times K_2&\to \mathcal L\times W^{1-\frac 1p,p}(\partial M)\\
((u,Y),(h,w))&\mapsto (\Phi,\Theta)[\Psi_{(g_\lambda,\pi_\lambda)}(u,Y)+(h,w)]
\end{align*}
are isomorphisms at the point $((1,0),(0,0))$. In fact, we have 
\begin{equation}\|D \hat{\mathcal P}_\lambda|_{((1,0),(0,0))}^{-1} \|_\mathrm{op}\le  C\label{inversebound}\end{equation}
and the Lipschitz constant of $D\hat{\mathcal P}_{\lambda}$ is bounded by $C$ on $B_{r_0}((1,0),(0,0))$. 
\end{lem}

For $\lambda=\infty$, the relevant differential is
\begin{equation}
D\hat{\mathcal P}_\infty|_{((1,0),(0,0))}((v,Z),(h,w))=(DT,D\Upsilon)|_{(1,0)}(u,Y)+(D\Phi,D\Theta)|_{(g,\pi)}(h,w).
\end{equation}
To construct $K_1$, we will require the following lemma.

\begin{lem}\label{Fredholm1}
$(DT_{(1,0)},DB_1|_{(1,0)})$ restricted to $\ker DB_2|_{(1,0)}\cap \ker DB_3|_{(1,0)}$ is Fredholm. 
\end{lem}
\begin{proof}
By Proposition \ref{elliptic} (1), we have the estimate 
  \[\|(v,Z)\|_{W^{2,p}_{-q}}\lesssim \|DT(v,Z)\|_\mathcal L+\|DB_1(v,Z)\|_{ W^{1-\frac 1p,p}}+\|(v,Z)\|_{L^p(U)}\]
  for every $(v,Z)\in \ker DB_2\cap \ker DB_3$. As mentioned in the proof of Proposition~\ref{elliptic}, it follows from
  \cite[Theorem 5.21]{Schechter} that $(DT,DB_1)$ is semi-Fredholm. It remains to show that $(DT,DB_1)[\ker DB_2\cap DB_3]$ has finite codimension. If not, then there exists an infinite-dimensional subspace $X\subset \mathcal L\times W^{1-\frac 1p,p}(\partial M)$ such that $(DT,DB_1)[\ker DB_2\cap DB_3]\cap X=0$. But that would imply that 
  \[(DT,DB)[T_{(1,0)}\mathcal C]\cap [X\times \{0\}\times\{0\}] =0,\]
  which is impossible since the full elliptic operator is Fredholm. 
\end{proof}

To construct $K_2$, we need to know that the lineariation of $(\Phi, \Theta)$ is surjective, which generalizes Proposition 3.1 in \cite{CorvinoSchoen}. See also the related work by Zhongshan An \cite{ZAn}.

\begin{prop}\label{prop1}
$(D\Phi,D\Theta)|_{(g,\pi)}:T_{(g,\pi)}\mathcal D\to \mathcal L\times W^{1-\frac 1p,p}(\partial M)$ is surjective.
\end{prop}
\begin{proof}
Since $D\Theta|_{(g,\pi)}:T_{(g,\pi)}\mathcal D\to W^{1-\frac 1p,p}(\partial M)$ is surjective (Lemma \ref{boundary_surjective}), it suffices to show that $D\Phi|_{(g,\pi)}:\ker(D\Theta_{(g,\pi)})\to \mathcal L$ is surjective. 

First we claim that 
\[D\Phi|_{(g,\pi)}[\ker D\Theta|_{(g,\pi)}]\subset\mathcal L\]
is closed and has finite codimension. It suffices to observe that
\begin{equation}DT|_{(1,0)}[\ker DB|_{(1,0)}]\subset D\Phi|_{(g,\pi)}[\ker D\Theta|_{(g,\pi)}],\label{inclusion1}\end{equation}
as the former is closed with finite codimension by repeating the argument of Lemma \ref{Fredholm1}.  

As a consequence of the Hahn--Banach theorem \cite[Corollary 1.8]{Brezis}, if $D\Phi|_{(g,\pi)}[\ker D\Theta|_{(g,\pi)}]\ne \mathcal L$, there is a nontrivial bounded linear functional $(\xi, V)\in \mathcal L^*= L^{p'}_{2+q-n}$ which annihilates it. By the inclusion \eqref{inclusion1}, $(\xi,V)$ annihilates $DT|_{(1,0)}(\ker DB|_{(1,0)})$. By considering arbitrary test data $(h,w)\in T_{(g,\pi)}\mathcal D$ compactly supported away from $\partial M$ (whence $(h,w)\in \ker DB|_{(1,0)}$), we see that $(\xi,V)$ solves the equation 
\begin{equation}D\Phi|_{(g,\pi)}^*(\xi,V)=0\label{dual}\end{equation}
in the sense of distributions. Arguing as in \cite[Appendix B]{HuangLee}, we conclude that $\xi$ and $V$ are $C^2$ in the interior and hence solve the equation classically. Note that the boundary behavior of $(\xi,V)$ is not needed for our argument. 

The result now follows from arguments in \cite{CorvinoSchoen}, as explained in detail in~\cite[Theorem 9.9]{Lee}. (The presence of a boundary is irrelevant to this part of the argument.) For the reader's convenience, we summarize the bas argument: The equations \eqref{dual} imply homogeneous Hessian-type equations for $(\xi,V)$, with coefficients decaying according to the asymptotic flatness assumption. Using this Hessian-type system, the $L^{p'}$ decay can be bootstrapped to become pointwise $C^1$ decay. Next, initial decay of $(\xi, V)$ then implies Hessian decay that is more than 2 orders faster, which gives improved decay on $(\xi, V)$ simply by (twice) integrating along coordinate rays to infinity. Bootstrapping in this way, $(\xi, V)$ must have infinite-order pointwise decay. From here, one can use a unique continuation argument (as in \cite{CorvinoSchoen})  to see that $(\xi, V)$ vanishes identically, or alternatively, as explained in the proof of~\cite[Theorem 9.9]{Lee}, for each $p\in M$ 
one can directly use the second-order system of ODEs satisfied by $(\xi, V)$ along a curve from $p$ to infinity to see that infinite-order decay implies vanishing at $p$. This ODE argument originated in~\cite[Lemma B.3]{HMM18}.
\end{proof}

With our preparatory results in place, we can construct $K_1$ and $K_2$.

\begin{proof}[Proof of Lemma \ref{main}]
Let $K_1\subset \ker DB_2\cap \ker DB_3$ (note that this intersection is closed in $T_{(1,0)}\mathcal C$ and hence is a Banach space) be a complementing subspace for the kernel of $(DT,DB_1)$ inside $\ker DB_2\cap \ker DB_3$. This subspace exists because the kernel is finite dimensional (Lemma \ref{Fredholm1}) and finite dimensional subspaces are always complemented. We then note that by the formula for $D\Upsilon$, for $(v,Z)\in K_1$,
\[D\Upsilon(v,Z)=DB_1(v,Z).\]
It follows that
\[(DT,D\Upsilon):K_1\to \mathcal L\times W^{1-\frac 1p,p}(\partial M)\]
is injective. Furthermore, its range $R=(DT,DB_1)[\ker DB_2\cap \ker DB_3]$ is a closed subspace of 
$\mathcal L\times W^{1-\frac 1p,p}(\partial M) $  with finite codimension by Lemma \ref{Fredholm1}.

Let $A$ be a finite-dimensional subspace of $\mathcal L\times W^{1-\frac 1p,p}(\partial M)$ which complements $R$ in $\mathcal L\times W^{1-\frac 1p,p}(\partial M)$. Using Proposition \ref{prop1}, we can find a finite-dimensional subspace $K'_2\subset T_{(g,\pi)}\mathcal D$ so that 
\[(D\Phi,D\Theta)|_{(g,\pi)}(K_2)= A.\]
Note that smooth, compactly supported sections of $(T^*M \odot T^* M)\times (TM\odot TM)$ are dense in
$T_{(g,\pi)}\mathcal D=
W^{2,p}_{-q}(T^*M \odot T^* M) \times W^{1,p}_{-q-1}(TM\odot TM)$, so we can find a finite-dimensional space $K_2$ made up of 
 smooth, compactly supported sections that closely approximates $K'_2$. By choosing a good enough approximation, the image of $K_2$ will be close enough to $A$ to still be complementary to $R$. With this choice of $K_1$ and $K_2$, $D\hat{\mathcal P}_\infty|_{((1,0),(0,0))}$ is an isomorphism.

Since $(g_\lambda, \pi_\lambda)\to(g_\infty, \pi_\infty)$ in $\mathcal D$, it follows from~\eqref{first} of Lemma \ref{AppendixA} that $D\hat{\mathcal P}_\lambda|_{((1,0),(0,0))}\to D\hat{\mathcal P}_\infty|_{((1,0),(0,0))}$ in operator norm as $\lambda\to\infty$. Therefore $D\hat{\mathcal P}_\lambda|_{((1,0),(0,0))}$ is also an isomorphism for sufficiently large $\lambda$, and its inverse satisfies \eqref{inversebound}. Finally, the Lipschitz constant bound follows from the Hessian bound \eqref{second} of Lemma \ref{AppendixA}. 
\end{proof}

We state the relevant standard elliptic regularity fact needed to establish the boundary regularity in Theorem~\ref{prescribed2}. The only subtlety is that we are not assuming $u$ to be $C^2$. The version here (for Sobolev $u$) can be read off from \cite[Theorem 6.4.8]{Morrey}.

\begin{lem}\label{Morrey}
Let $\Omega\subset\Bbb R^n$ be a bounded domain with $C^{2,\alpha}$ boundary, $\alpha\in (0,1)$. Suppose that  $a^{ij}\in C^{0,\alpha}(\overline\Omega)$ is positive definite, and suppose that  $f\in C^{0,\alpha}(\overline\Omega)$ and $g\in C^{2,\alpha}(\partial\Omega)$. If $u\in W^{2,p}(\Omega)$ ($p>n$) is a strong solution of 
\begin{align*}
    a^{ij}\partial_i\partial_ ju&=f\quad\text{in $\Omega$},\\
      u&= g\quad\text{on $\partial\Omega$},
\end{align*}
then $u\in C^{2,\alpha}(\overline\Omega).$

If instead $\beta^i\in C^{1,\alpha}(\partial\Omega)$ is an oblique vector field and $g\in C^{1,\alpha}(\partial\Omega)$ and $u\in W^{2,p}(\Omega)$ ($p>n$) satisfies 
\begin{align*}
    a^{ij}\partial_i\partial_ ju&=f\quad\text{in $\Omega$},\\
    \beta^i \partial_i u&= g\quad\text{on $\partial\Omega$},
\end{align*}
then $u\in C^{2,\alpha}(\overline\Omega)$.
\end{lem}

We now use the inverse function theorem to prove Theorem \ref{prescribed2}. The proof is a boundary version of the argument given for \cite[Theorem 9.10 and Proposition 9.11]{Lee}, which itself is based on \cite{CorvinoSchoen}. 

\begin{proof}[Proof of Theorem \ref{prescribed2}]
By Lemma \ref{main} and the inverse function theorem for Banach spaces (see in particular the ``quantitative" version \cite[Theorem A.43]{Lee}), there exists a constant $C$ such that for $r>0$ sufficiently small and $\lambda$ sufficiently large, $\hat{\mathcal P}_\lambda^{-1}$ exists and maps $B_r(2\mu_\lambda,J_\lambda,\theta)\subset \mathcal L\times W^{1-\frac 1p,p}(\partial M) $
into $B_{Cr}((1,0),(0,0))\subset [(1,0)+K_1]\times K_2$, where $(\mu_\lambda, J_\lambda)$ are the constraints of $(g_\lambda, \pi_\lambda)$. So if $(\tilde \mu,\tilde J,\tilde \theta)$ satisfy the hypotheses of the theorem with 
\[\|(2\tilde \mu,\tilde J,\tilde\theta)-(2\mu,J,\theta)\|_{\mathcal L\times W^{1-\frac 1p,p}}<\frac r2\] and $\lambda$ is sufficiently large that $\|(2\mu_\lambda,J_\lambda)-(2\mu,J)\|_\mathcal L<\frac r2$,
then there exist $\alpha_\lambda\in [(1,0)+K_1]\times K_2$ such that $\hat{\mathcal P}_\lambda(\alpha_\lambda)=(2\tilde \mu,\tilde J,\tilde\theta)$, and
\[\|\alpha_\lambda -((1,0),(0,0))\|_{\mathcal C\times T_{(g,\pi)}\mathcal D}\le C\|(2\tilde \mu,\tilde J,\tilde\theta)-(2\mu,J,\theta)\|_{\mathcal L\times W^{1-\frac 1p,p}}\le Cr.\]
By choosing $r$ smaller (and hence also $\lambda$ larger) depending on the constant in Morrey's inequality, we can ensure $|u_\lambda-1|< 1$ everywhere, so that it is a valid conformal factor. Set $\delta=\frac r2$. 

Having made all of these choices, write $\alpha_\lambda=((u,Y), (h,w))$. We claim that the initial data $(\tilde g,\tilde\pi)=\Psi_{(g_\lambda, \pi_\lambda)}(u, Y)+(h,w)$ is the desired solution in the conclusion of Theorem~\ref{prescribed2}. By construction, it has the desired constraints $(\tilde \mu,\tilde J)$ and outer/inner null expansion $\tilde\theta$ on $\partial^\pm M$, and it satisfies the desired estimate
\[\|(\tilde g,\tilde\pi)-(g,\pi)\|_\mathcal D\le C\|(\tilde \mu,\tilde J,\tilde\theta)-(\mu,J,\theta)\|_{\mathcal L\times W^{1-\frac 1p,p}}.\]
And since $(h_\lambda,w_\lambda)$ is compactly supported, 
\[(\tilde g,\tilde\pi)=\Psi_{(\overline g,0)}(u_\lambda,Y_\lambda)\]
for $|x|\ge 2\lambda$, as desired. 

It only remains to show that show that if $(\tilde\mu,\tilde J)$ is $C^{0,\alpha}_{-n-1-\ve}$ up to the boundary and $\tilde\theta$ is $C^{1,\alpha}$, then $(\tilde g,\tilde\pi)$ is $C^{2,\alpha}_{2-n}\times C^{1,\alpha}_{1-n}$ up to the boundary and has harmonic asymptotics. Since $(h,w)$ is compactly supported and smooth, it suffices to show that $(u,Y)$ satisfies the asymptotic expansion in Definition~\ref{harmasymp} and is $C^{2,\alpha}$ up to the boundary. The asymptotic expansion is well-known from earlier references such as~\cite[Lemma 9.8]{Lee}, but here we carefully account for the presence of the $(h,w)$ term and the boundary in order to prove regularity up to the boundary.

To prove the desired result, we re-write 
\[\hat{\mathcal P}_\lambda ((u,Y),(h,w))=(2\tilde \mu,\tilde J,\tilde\theta)\] 
as a set of linear elliptic equations in $(u, Y)$, viewing the nonlinearities as either coefficients or nonhomogenous terms. More precisely, after we choose local coordinates (which is fine since we are proving a local regularity result now), we will  have equations for $u, Y_1,\ldots, Y_n$ of the form described in Lemma~\ref{Morrey} above. We will now explain this in detail. 

By the Morrey embedding theorem, we already know that $(u,Y)\in C^{1,\alpha}_{\mathrm{loc}}$ on $M$. We first examine the $\tilde\mu$ equation as our equation for $u$. First, the non-scalar curvature terms of $\tilde\mu$ are $C^{0,\alpha}_{\mathrm{loc}}$ and therefore can be viewed as part of the $C^{0,\alpha}_{\mathrm{loc}}$ nonhomogenous term. Moreover, using~\cite[Exercise 1.2]{Lee} or otherwise, we can re-write 
\begin{equation}\label{tilde_operator}
R(u^s g_\lambda + h) = 
su^{s-1}(\tilde g^{ik} \tilde g^{jl}-\tilde g^{ij}\tilde g^{kl})  (g_{\lambda})_{kl} \partial_i \partial_j u + (\text{terms in }C^{0,\alpha}_{\mathrm{loc}})
\end{equation}
Let $U$ be a bounded set large enough so that every element of $K_2$ vanishes outside $U$. So in the complement of $U$, where $\tilde{g}$ is just a conformal change, this becomes
\[R(u^s g_\lambda + h) = -s(n-1)u^{-s-1}(g_{\lambda})^{ij} \partial_i \partial_j u+  (\text{terms in }C^{0,\alpha}_{\mathrm{loc}}),
\]
whose coefficient matrix is obviously negative definite. 
Meanwhile,  $g_\lambda=g$ on $U$ for sufficiently large $\lambda$, so for $\delta$ sufficiently small, \eqref{tilde_operator} will be strictly elliptic because $|u-1|$ and $h$ can be made uniformly small enough so that the second-order coefficients in~\eqref{tilde_operator} can be made uniformly close to
\[
s(   g^{ik} g^{jl}- g^{ij} g^{kl})  g_{kl} = -s(n-1)g^{ij},
\]
which we know is negative definite. In either case, $u$ solves an elliptic equation of the type described in~Lemma~\ref{Morrey}. 



For the $\tilde{J}$ equations, we compute 
\begin{equation*}
    \Div_{\tilde g}\tilde \pi^i=u^{-\frac 32 s}\left[-\tfrac 32 s\frac{\nabla_j u}{u} (\pi_\lambda^{ij}+\mathfrak L_{g_\lambda} Y^{ij})+\tilde \nabla_j(\pi_\lambda^{ij}+\mathfrak L_{g_\lambda}Y^{ij}+w^{ij})\right].
    \end{equation*}
Again, the lower order terms are clearly $C^{0,\alpha}_{\mathrm{loc}}$, so we have
\begin{equation}\label{inC0alpha}
    \tilde \nabla_j \mathfrak L_{g_\lambda} Y^{ij} \in C^{0,\alpha}_{\mathrm{loc}}.
\end{equation} 
Meanwhile, by looking at the first order terms of 
$\mathfrak L_{g_\lambda}Y$, we have 
\[\mathfrak L_{g_\lambda}Y^{ij} =
g_\lambda^{ik}\partial_k Y^j+g_\lambda^{jk}\partial_k Y^i- \partial_k Y^k g_\lambda^{ij}+(\text{terms in }C^{0,\alpha}_{\mathrm{loc}}).\]
By taking $\partial_j$, we can see that 
\begin{align*}
    \tilde \nabla_j \mathfrak L_{g_\lambda} Y^{ij}&=
g_\lambda^{ik}\partial_j\partial_k Y^j+g_\lambda^{jk}\partial_j\partial_k Y^i- \partial_j\partial_k Y^k g_\lambda^{ij} +(\text{terms in }C^{0,\alpha}_{\mathrm{loc}})\\
&=g^{jk}_\lambda\partial_k\partial_j Y^i +(\text{terms in }C^{0,\alpha}_{\mathrm{loc}}).
\end{align*}
Combining this with~\eqref{inC0alpha}, we see that 
\[ g^{jk}_\lambda\partial_k\partial_j Y^i\in C^{0,\alpha}_{\mathrm{loc}},\]
where $g_\lambda$ is positive definite, so each $Y^i$ also satisfies an elliptic equation of the of the type described in~Lemma~\ref{Morrey}. 

The only thing left to check is that $u, Y^1,\ldots, Y^n$ satisfy boundary conditions of the type described in Lemma~\ref{Morrey}. 
We have stipulated that $\Theta(\tilde g,\tilde \pi)=\tilde\theta$, which is only one boundary condition. However, the space $K_1$ in the definition of $\hat{\mathcal P}$ contains the other $n$ boundary conditions we need. Indeed, since $(u,Y)\in K_1$, then by definition 
\begin{align*}
    Y^\top &=0\\
   HY^n -\frac{\partial Y^n}{\partial x^n} &=0,
\end{align*}
where we have expressed the second condition in an almost-Fermi coordinate system. (We may do this since regularity is a local property.) So we see that $Y^1,\ldots, Y^{n-1}$ satisfy the Dirichlet boundary condition, and we claim that $Y^n$ and $u$ satisfy Neumann-type conditions of the type described in Lemma \ref{Morrey}. For $Y^n$, this is immediate from observing that 
$HY^{n-1}$ is $C^{1,\alpha}$ up to the boundary. Hence Lemma \ref{Morrey} implies that $Y$ is $C^{2,\alpha}$ up to the boundary.  

Using this upgraded regularity for $Y$, we can interpret $\Theta(\tilde g,\tilde \pi)=\tilde\theta$ as a Neumann-type boundary condition for $u$.
We clearly have $-\tilde\pi^{ij}\tilde\nu_i\tilde\nu_j\in C^{1,\alpha}(\partial M)$, so we just have to investigate the mean curvature of $\partial M$ with respect to $\tilde{g}$, which we denote by $H_{\tilde g}$. As in the proof of Lemma \ref{theta_regularity}, we write $\partial M$ as a regular level set of a smooth function $f$, so that 
\[H_{\tilde g}=\gamma^{ij}
\frac{1}{|\nabla f|_{\tilde g}}(\partial_{ij}f-\tilde \Gamma^k_{ij}\partial_k f),\]
where $\gamma^{ij}=\tilde g^{ij}-\tilde \nu^i\tilde \nu^j$.
By focusing only on the terms with derivatives of $u$ (in the Christoffel symbols), we have
\[(-2\gamma^{ij}\tilde g^{kl} + \gamma^{lj}\tilde g^{ik}) 
\tilde \nu_k su^{s-1} g_{jl}\partial_i u 
 \in C^{1,\alpha}(\partial M),
\]
which is an equation of the form $\beta^i\partial_i u\in C^{1,\alpha}(\partial M)$, where each $\beta^i\in C^{1,\alpha}(\partial M)$ as well.
 The only thing left to check is that $\beta^i$ is not tangent to $\partial M$. To see this, observe that when $h=0$, we have $\beta^i = (n-1)su^{-1}\nu^i$, and therefore $\beta^i$ is oblique for sufficiently small $h$.
\end{proof}

\subsection{Dominant energy condition density theorem} We begin with a precise re-statement of Theorem \ref{DECdensity}:

\begin{thm}[Density theorem for DEC] \label{DECdensity2}
Let $(M^n,g,\pi)$ be a complete asymptotically flat initial data set satisfying the dominant energy condition $\mu \ge |J|$ and with compact boundary $\partial M$ having outer null expansion $\theta$ on $\partial^+ M$ and inner null expansion $\theta$ on $\partial^- M$ . Let $p>n$ and $\frac{n-2}{2}<q<n-2$ be strictly less than the decay rate of $(g,\pi)$. For any $\ve>0$ there exists a constant $\delta>0$ so that the following is true:

For any $\tilde\theta \in C^{1,\alpha}(\partial M)$ satisfying $\|\tilde\theta -\theta\|_{W^{1-\frac 1p,p}}<\delta$ there exists an asymptotically flat initial data set $(\tilde g,\tilde \pi)$, $C^{2,\alpha}\times C^{1,\alpha}$-regular up to the boundary, also satisfying the dominant energy condition, such that $(\tilde g,\tilde \pi)$ has harmonic asymptotics in each end of $M$, the new outer/inner null expansion on $\partial^\pm M$ is $\tilde\theta$, and the new data set satisfies \[\|(\tilde g,\tilde\pi)-(g,\pi)\|_{\mathcal D}<\ve\quad\text{and}\quad\|(\tilde \mu,\tilde J)-(\mu,J)\|_{L^1}<\ve.\]

Furthermore, we can choose $(\tilde g,\tilde \pi)$ such that the \emph{strict} dominant energy condition holds, $\tilde \mu>|\tilde J|$. Simultaneously, $(\tilde \mu,\tilde J)$ may be chosen to decay as fast as we like in the sense that if $f$ is any positive smooth function, then we can demand $\tilde \mu +|\tilde J|\le f(|x|)$ on the end. 

Alternatively, we can choose $(\tilde g,\tilde \pi)$ to be vacuum outside a compact set, that is, $\tilde\mu =|\tilde J|=0$ outside a compact set. 
\end{thm}

Given a fixed initial data set $(M^n,g,\pi)$, the \emph{modified constraint operator} $\overline\Phi_{(g,\pi)}$ is defined by 
\[\overline\Phi_{(g,\pi)}(\gamma,\tau)=\Phi(\gamma,\tau)+(0,\tfrac 12 g^{ij}\gamma_{jk}J^k)\]
for $(\gamma,\tau)\in\mathcal D$.

\begin{lem}[Corvino--Huang \cite{CH}]\label{CH}
Let $(g,\pi)$ and $(\tilde g,\tilde \pi)$ be initial data, and assume that 
\[\overline\Phi_{(g,\pi)}(\tilde g,\tilde\pi)-\overline\Phi_{(g,\pi)}(g,\pi)=(2\psi,0)\]
for some function $\psi$. If additionally $|\tilde g-g|_{g}\le 3$, then  
\[|\tilde J|_{\tilde g}\le |J|_{g}.\]
\end{lem}

The linearization of the modified constraint operator at $(g,\pi)$ is given by 
\[D\overline\Phi_{(g,\pi)}(h,w)= D\Phi|_{(g,\pi)}(h,w)+(0,\tfrac12 g^{ij}h_{jk}J^k) \]
The addtional zeroth order term $(0,\tfrac12 g^{ij}h_{jk}J^k)$ does not affect the proof of Proposition \ref{prop1}, and so we obtain the surjectivity result.
\begin{prop}\label{modified_surj}
$(D\overline\Phi_{(g,\pi)},D\Theta)|_{(g,\pi)}:T_{(g,\pi)}\mathcal D\to \mathcal L\times W^{1-\frac 1p,p}(\partial M)$ is surjective.
\end{prop} 
By Proposition \ref{modified_surj}, we can use the construction of Lemma \ref{main} to construct the analogous subspaces $K_1,K_2$. Using this and the inverse function theorem argument in Theorem \ref{prescribed2}, we show that it is possible to perturb the initial data to strict DEC. Proposition \ref{strictDEC} does not attempt to produce harmonic asymptotics, which we get to in Theorem \ref{DECdensity2}. 

\begin{prop}[Perturbing to strict DEC]\label{strictDEC}
Let $(M^n,g,\pi)$ be a complete asymptotically flat initial data set satisfying the dominant energy condition $\mu \ge |J|$ and with $\Theta(g,\pi)=\theta$. Let $p>n$ and $\frac{n-2}{2}<q<n-2$ be strictly less than the decay rate of $(g,\pi)$. For any $\ve>0$ there exist constants $\delta>0$ and $\gamma>0$ so that the following is true:

For any $\tilde\theta \in C^{1,\alpha}(\partial M)$ satisfying $\|\tilde\theta -\theta\|_{W^{1-\frac 1p,p}}<\delta$ there exists an asymptotically flat initial data set $(\tilde g,\tilde \pi)$, $C^{2,\alpha}\times C^{1,\alpha}$-regular up to the boundary, with  $\Theta(\tilde g,\tilde \pi)=\tilde \theta$, that satisfies the following ``uniform'' strict dominant energy condition
\[\tilde \mu >(1+\gamma)|\tilde J|_{\tilde g},\]
as well as the estimates
\[\|(\tilde g,\tilde\pi)-(g,\pi)\|_{\mathcal D}<\ve\quad\text{and}\quad \|(\tilde \mu,\tilde J)-(\mu,J)\|_{L^1}<\ve.\] 
\end{prop}
\begin{proof}
Let $f$ be a smooth positive function on $M$ decaying exponentially at infinity.  By essentially repeating the proof of Lemma \ref{main}, we can construct subspaces $K_1\subset \ker DB_2\cap \ker DB_3$ and $K_2\subset T_{(g,\pi)}\mathcal D$ (consisting of compactly supported smooth functions) to define an operator
\begin{align*}
\hat{\mathcal P}:[(1,0)+K_1]\times K_2&\to \mathcal L\times W^{1-\frac 1p,p}(\partial M)\\
((u,Y),(h,w))&\mapsto (\overline\Phi_{(g,\pi)},\Theta)[\Psi_{(g,\pi)}(u,Y)+(h,w)]
\end{align*}
whose differential at $((1,0),(0,0))$ is an isomorphism. We now proceed as in the proof of Theorem~\ref{prescribed2}. In particular, we use the inverse function theorem \cite[Theorem A.43]{Lee} to solve 
\begin{equation}\label{strictPDE}
    \hat{\mathcal{P}}((u_t,Y_t),(h_t,w_t))=\left( \overline{\Phi}_{(g,\pi)}(g,\pi)+(2t( f+|J|_g),0), \,\tilde{\theta} \right)
\end{equation} for $((u_t,Y_t),(h_t,w_t))$,
which is possible for sufficiently small $\delta>0$ and $t>0$. 

Arguing as in the proof of Theorem \ref{prescribed2}, we can show that $(u_t,Y_t)$ decays enough so that $(\tilde g_t, \tilde\pi_t)=\Psi_{(g,\pi)}(u_t,Y_t)+(h_t,w_t)$ is asymptotically flat, and $(u_t, Y_t)$ is $C^{2,\alpha}$ up to the boundary. 

We now see about the other claims in the proposition. First, we claim that $(\tilde g_t,\tilde\pi_t)$ satisfies the strict DEC for $t$ small enough. By the Sobolev inequality, we may assume $\tilde g_t-g$ is uniformly small pointwise. Therefore, by Lemma \ref{CH}, 
\[\tilde \mu_t =\mu+t(f+|J|_g)>\mu +t|J|_g \ge (1+t)|J|_g\ge (1+t)|\tilde J_t|_{\tilde g_t}.\] Finally, we need to show that $(\tilde \mu_t,\tilde J_t)\to (\mu,J)$ in $L^1$. Since $f+|J|_g\in L^1$, it is clear from the equation $\tilde \mu_t=\mu+t(f+|J|_g)$ that $\tilde \mu_t\to \mu$. For the momentum we have $\tilde J_t^i-J^i=\frac 12 g^{ij}(\tilde g_t-g)_{jk}J^k$, but $\tilde g_t\to g$ uniformly, so $\tilde J_t\to J$. We take $\gamma$ to be the $t$ satisfying these conditions, which completes the proof.  
\end{proof}

Finally, we prove Theorem \ref{DECdensity2}.
\begin{proof}[Proof of Theorem \ref{DECdensity2}] First, perturb the data set according to Proposition \ref{strictDEC} to ensure the strict DEC holds in the form $\mu >(1+\gamma)|J|_g$  everywhere on $M$. Call the perturbed data set $(g,\pi)$ for simplicity. Let $\chi_\lambda$ be the family of cutoff functions used in Theorem \ref{prescribed2} and $f$ a rapidly decaying positive function on $M$. By Theorem \ref{prescribed2}, we may construct a data set $(\tilde g_\lambda,\tilde \pi_\lambda)$ with harmonic asymptotics satisfying 
\[(\Phi,\Theta)(\tilde g_\lambda,\tilde \pi_\lambda)=\left(\chi_\lambda \Phi(g,\pi)+\tfrac 2\lambda (f,0),\,\tilde\theta\right)\]
if $\lambda$ is large enough and $\delta$ is small enough. Note that since we can choose $f$ to decay as rapidly as we like, our prescribed $(\tilde\mu,\tilde J)$ will certainly lie in $C^{0,\alpha}_{-n-1-\ve}$.

The only thing left to check is that the strict DEC is satisfied:
\begin{align*}
    \tilde \mu_\lambda-\frac 1\lambda f&=\chi_\lambda\mu\\
    &\ge \chi_\lambda(1+\gamma)|J|_g\\
    &\ge \chi_\lambda \left(|J|_{\tilde g_\lambda} -|\tilde g_\lambda -g|^\frac 12_g|J|_g+\gamma |J|_g\right)\\
    &= |\tilde J_\lambda|_{\tilde g_\lambda} +\chi_\lambda |J|_g(\gamma-|\tilde g_\lambda -g|^\frac 12_g).
\end{align*}
For $\lambda$ large, $|\tilde g_\lambda -g|^\frac 12_g<\gamma$ by the estimates in Theorem \ref{prescribed2}, so we then have 
\[\tilde \mu_\lambda -|\tilde J_\lambda|_{\tilde g_\lambda}\ge \frac 1\lambda f,\]
which implies the strict DEC. Note that $f$ also controls the decay of the DEC scalar $\tilde \mu_\lambda -|\tilde J_\lambda|_{\tilde g_\lambda}$, as it is identically equal to $\frac 1\lambda f$ for $|x|\ge 2\lambda$. 

If we instead wish to prescribe vacuum outside a compact set, we perform the same argument as above but with $f\equiv 0$. 
\end{proof}

\section{Positive mass theorem with boundary}\label{S4}

\subsection{Proof of the inequality $E\ge |P|$}\label{S4.inequality} 


In this subsection we explain how Theorem~\ref{PMT} follows from combining Theorem~\ref{DECdensity2}  with the proof of the boundaryless case from~\cite{EHLS}. 
Suppose that there exists a complete asymptotically flat initial data set $(M, g, \pi)$ satisfying the DEC, whose compact boundary is made up of components which are either weakly outer trapped 
($\theta^+\le0$) or weakly inner untrapped ($\theta^-\ge0$) with respect to the normal pointing into $M$, such that $E<|P|$. By Theorem \ref{DECdensity2}, we can perturb $(g,\pi)$ to new initial data $(\tilde{g}, \tilde{\pi})$ that has harmonic asymptotics, satisfies the strict DEC, and has compact boundary made up of components that have either $\theta^+<0$ or $\theta^->0$, while maintaining the inequality  $\tilde{E}<|\tilde{P}|$. From here the exact same argument as in \cite{EHLS} (after the application of the density theorem there) results in a contradiction. The only thing to note is that the boundary acts as a barrier for the MOTS ($\theta^+=0$ hypersurfaces) that are constructed in the proof. This part of the proof is also identical to the reasoning used in \cite{GallowayLee}. (Note that the H\"{o}lder decay assumption on $(\mu, J)$ in \cite{EHLS} is unnecessary, as can be seen from our proof and was observed in \cite{Lee, GallowayLee}.) 

To be more precise, in \cite{EHLS}, one seeks to construct a stable MOTS hypersurface in $M$ with prescribed boundary equal to a large sphere $\Gamma^{n-2}$ of constant height on a large cylinder $C$ (with smoothed corners). Theorem 1.1 of \cite{E09} states that this is possible if one can find a compact $\Omega$ such that $\Gamma\subset\partial\Omega$ divides $\partial\Omega$ into $\partial_1\Omega$ and $\partial_2\Omega$ such that $\theta^+_{\partial_1\Omega}>0$ with respect to the normal pointing \emph{out of} $\Omega$, and $\theta^+_{\partial_2\Omega}<0$ with respect to the normal pointing \emph{into} $\Omega$. If $M$ has one end and no boundary, then we choose $\Omega$ to be the region enclosed by $C$, $\partial_1\Omega$ to be the part of $C$ lying above $\Gamma$, and $\partial_2\Omega$ to be the part of $C$ lying below $\Gamma$. Harmonic asymptotics guarantee that if $C$ is big enough, these choices satisfy the hypotheses on $\theta^+$ needed to apply \cite[Theorem 1.1]{E09}. If there are multiple ends, then we choose $\Omega$ to be enclosed by $C$ in the end of interest and large celestial spheres in all other ends. Those celestial spheres have $\theta^+<0$ with respect to the normal pointing into $\Omega$, and hence those spheres can be included as part of $\partial_2\Omega$. Finally, we come to the case of interest where $M$ has a boundary. We define $\Omega$ the same way, except now we can treat any $\theta^+<0$ components of $\partial M$ as part of $\partial_2 \Omega$ while treating any $\theta^->0$ components of $\partial M$ as part of $\partial_1\Omega$, because the condition $\theta^->0$ with respect to the normal pointing into $M$ is equivalent to the condition $\theta^+>0$ with respect to the normal pointing out of $\Omega$. $\hfill\qed$

\subsection{The equality case $E=|P|$}\label{S4.1}

Here we explain how the arguments in \cite{HuangLee} can be adapted to handle a boundary, using the results of this paper. Explicitly, we prove the following theorem, which constitutes the first part of the proof of Theorem~\ref{strictPMT}.

\begin{thm}\label{ADMnull}
Let $n\ge 3$, $p>n$, $q> \tfrac{n-2}{2}$, and $0<\alpha<1$, and assume that 
\begin{align}\label{equation:extra}
	q+\alpha >n-2.
\end{align}
Suppose that $(M^n, g, \pi)$ is a complete asymptotically flat initial data set with boundary (in the sense defined in Definition~\ref{AFdata}) with the stronger decay assumption that 
\begin{align}\label{equation:AF}
	g - \overline{g}& \in C^{2,\alpha}_{-q}(T^*M\odot T^*M) \\
	\pi&\in  C^{1,\alpha}_{-1-q}(TM\odot TM).
\end{align}
Then if $(M, g, \pi)$ satisfies the DEC, each component of $\partial M$ is either weakly outer trapped or weakly inner untrapped, and its ADM energy-momentum satisfies $E=|P|$, then $E=|P|=0$.
\end{thm}

As explained in \cite{HuangLee2}, in the case without boundary, the theorem is actually false without the stronger decay assumption.

\begin{proof}
Assume that $(M, g, \pi)$ satisfies the hypotheses of Theorem~\ref{ADMnull}. The basic strategy is the following: 
Using the first part of Theorem~\ref{PMT}, we can see that $(g,\pi)$ minimizes a ``modified Regge--Teitelboim Hamiltonian'' among all nearby initial data sets that have the same values of $(\overline{\Phi}_{(g,\pi)}, \Theta)$. By Proposition \ref{modified_surj}, $(D\overline{\Phi}_{(g,\pi)}, D\Theta)|_{(g,\pi)}$ is surjective, and hence we can apply Lagrange multipliers. These Lagrange multipliers give rise to a solution $(f,X)$ of the adjoint equations $D\overline{\Phi}_{(g,\pi)}|_{(g,\pi)}^* (f, X) =0$ such that $(f, X)$ is asymptotic to the constant $(E, -2P)$. Once we have that, a result of Beig and Chru\'{s}ciel~\cite{BeigChrusciel} (see also~\cite[Theorem A.2]{HuangLee}) implies that $E=|P|=0$. 

For the analysis that follows, select $q\in (\tfrac{n-2}{2}, n-2)$ that is smaller than the $q$ in statement of Theorem~\ref{ADMnull}.
Let $(f_0, X_0)$ be a function and a vector field on $M$ such that $(f_0, X_0)$ is smooth, supported in the asymptotically flat coordinate chart, and exactly equal to the constant $(E,-2P)$ outside some compact set, where $(E, P)$ denotes the fixed ADM energy-momentum of $(g,\pi)$. We define the \emph{modified Regge--Teitelboim Hamiltonian} $\mathcal{H} :\mathcal{D} \to \mathbb{R} $ corresponding to $(g,\pi)$ by 
\begin{align}\label{equation:functional}
	\mathcal{H} (\gamma, \tau) =2(n-1)\omega_{n-1} \left[ E\cdot E(\gamma,\tau) -P\cdot P(\gamma, \tau)\right] - \int_M \overline{\Phi}_{(g,\pi)}(\gamma, \tau) \cdot (f_0, X_0)\, d\mu_g,
\end{align}
for all $(\gamma,\tau)\in\mathcal{D}$,
where the volume measure $d\mu_g$ and the inner product in the integral are both with respect to~$g$. Although $E(\gamma,\tau)$, $P(\gamma, \tau)$, and the integral need not exist for elements $(\gamma, \tau)\in \mathcal{D}$ whose constraints are not integrable, the expression $\mathcal{H} (\gamma, \tau)$ can be given meaning by using the alternative formula:
\begin{align}\label{equation:H}
\begin{split}
	\mathcal{H} (\gamma, \tau) &=\int_M \left[\left(\Div_g [\Div_{\overline g} \gamma - d(\mathrm{tr}_{\overline g} \gamma) ], \Div_g \tau\right)- \Phi(\gamma, \tau)- \left(0, \tfrac{1}{2} \gamma \cdot J\right) \right] \cdot (f_0, X_0)\, d\mu_g\\
	&\quad + \int_M \left( [\Div_{\overline g} \gamma - d(\mathrm{tr}_{\overline g} \gamma) ], \tau \right)\cdot (\nabla f_0, \nabla X_0) \, d\mu_g,
	\end{split}
\end{align}
where $\overline g$ is a globally defined background metric that is Euclidean in the asymptotically flat end. As in~\cite{HuangLee}, we can compute the linearization $D\mathcal{H}|_{(g,\pi)}: T_{(g,\pi)}\mathcal{D} \to \mathbb{R}$ to be
\begin{equation}\label{equation:DH}
	 D\mathcal{H}|_{(g,\pi)} (h, w) =- \int_M (h, w) \cdot (D\overline{\Phi}_{(g,\pi)}|_{(g,\pi)})^*(f_0, X_0) \, d\mu_g,
\end{equation}
for all $(h,w)\in T_{(g,\pi)}\mathcal{D}$. Note that $\partial M$ can be ignored in all of these formulae because $(f_0, X_0)$ vanishes near $\partial M$.

Next we define a constraint space
\[
	\mathfrak{C}_{(g,\pi)} = \left\{ (\gamma, \tau) \in \mathcal{D}: \overline{\Phi}_{(g,\pi)}(\gamma, \tau) = \overline{\Phi}_{(g,\pi)}(g, \pi)\text{ and } \Theta(\gamma, \tau) = \Theta(g, \pi)\right\}.
\]
By our assumptions on $(g, \pi)$, each data set $(\gamma,\tau)$ in $\mathfrak{C}_{(g,\pi)}$ has weakly outer trapped or inner untrapped boundary components, and thanks to Lemma~\ref{CH}, it also satisfies the DEC. Moreover, since $(\gamma,\tau)$ has the same modified constraints as $(g,\pi)$, it also follows that its constraints are integrable. 
Then Theorem~\ref{PMT} implies that if $(\gamma,\tau)$ is near enough to $(g,\pi)$, then $E(\gamma,\tau)\ge |P(\gamma,\tau)|$. (We will discuss this more below. See Lemma~\ref{SobolevPMT}.) From this, we see that $(g,\pi)$ locally minimizes $\mathcal{H}$ on $\mathfrak{C}_{(g,\pi)}$ since 
 \begin{multline*}   
	E\cdot E(\gamma,\tau) - P \cdot P(\gamma, \tau) \ge E \cdot E(\gamma,\tau) - |P||P(\gamma,\tau)| \\
	= E( E(\gamma,\tau) - |P(\gamma, \tau)|) \ge 0 = E\cdot E(g, \pi) - P \cdot P(g, \pi),
\end{multline*}
and the integral term in \eqref{equation:functional} is constant over $\mathfrak{C}_{(g,\pi)}$.

In other words, $(g,\pi)$ locally minimizes $\mathcal{H}$ over a level set of $(\overline{\Phi}_{(g,\pi)} , \Theta):\mathcal{D}\to \mathcal L\times W^{1-\frac 1p,p}(\partial M)$, so by surjectivity of its linearization (Proposition \ref{modified_surj}), there exist Lagrange multipliers (see \cite[Appendix D]{HuangLee}) $(f_1, X_1)\in \mathcal{L}^*$, and $\lambda\in W^{1-\frac 1p,p}(\partial M)^*$ such that for all $(h,w)\in T_{(g,\pi)}\mathcal{D}$, 
\[
		\left. D\mathcal{H} \right|_{(g, \pi)} (h, w)  =\int_M (f_1, X_1) \cdot D\overline{\Phi}_{(g,\pi)}|_{(g,\pi)} (h, w) \, d\mu_g + \int_{\partial M} \lambda \cdot D\Theta|_{(g,\pi)}(h,w).
\]
Combining this with~\eqref{equation:DH} and by choosing $(h,w)$ to be arbitrary smooth test data that is compactly supported away from $\partial M$, we see that $(f_1, X_1)$ must be a solution (in the distributional sense) of 
\[ D\overline{\Phi}_{(g,\pi)}|^*_{(g,\pi)} (f_1, X_1) =  - D\overline{\Phi}_{(g,\pi)}|^*_{(g,\pi)} (f_0, X_0) 
\]
in the interior of $M$. As argued in the proof of Proposition~\ref{prop1}, ellipticity of $DT|_{(1,0)}$ implies that $(f_1, X_1)$ is actually smooth in the interior of $M$. Moreover, using the initial decay from being in $\mathcal{L}^*$ together with elliptic estimates, it follows that $(f_1, X_1)$ has $C^{2,\alpha}_{-q}$ decay. (See~\cite[Proposition B.4]{HuangLee} for details.) 
Thus $(f,X):=(f_0, X_0)+(f_1, X_1)$ solves
\[ D\overline{\Phi}_{(g,\pi)}|^*_{(g,\pi)} (f, X) = 0
\]
in the interior of $M$, and $(f, X)-(E, -2P)$ has $C^{2,\alpha}_{-q}$ decay. The result now follows from Theorem A.2 of \cite{HuangLee}. Note that it does not matter what what $(f, X)$ does near the boundary $\partial M$ since Theorem A.2 of \cite{HuangLee} is only a statement about asymptotics and makes no global assumptions. 
\end{proof}

There is one step in the proof above that requires further justification. We claimed that elements $(\gamma,\tau)$ of $\mathfrak{C}_{(g,\pi)}$ must satisfy $E(\gamma,\tau)\le |P(\gamma, \tau)|$, but the problem is that $(\gamma, \tau)$ may only have Sobolev regularity and decay, but our positive mass theorem (Theorem~\ref{PMT}) requires at least $C^{2,\alpha}\times C^{1,\alpha}$ local regularity as well as pointwise decay. Although we do not have a positive mass theorem for initial data in $\mathcal D$ with integrable constraints, we can at least prove it for data that is \emph{near} the smooth data $(g,\pi)$. This is the same idea that was used in \cite[Theorem 4.1]{HuangLee}.

\begin{lem}[Sobolev version of positive mass inequality, with boundary]\label{SobolevPMT}
Let $3\le n\le 7$, and let $(M^n, g, \pi)$ be a complete asymptotically flat manifold, as in Definition \ref{AFdata}, satsifying the DEC and  with compact boundary such that each component is either weakly outer trapped or weakly inner untrapped. Let $p>n$ and let $q\in (\tfrac{n-2}{2}, n-2)$ be smaller than the assumed  asymptotic decay rate of $(g,\pi)$. Then there is an open ball $U\subset\mathcal{D}$ containing $(g, \pi)$ such that if $(\gamma, \tau)\in U$, $\overline{\Phi}_{(g,\pi)}(\gamma,\tau) = \overline{\Phi}_{(g,\pi)}(g, \pi)$, and $\Theta(\gamma, \tau)=\Theta(g,\pi)$, then 
\[
	E(\gamma, \tau) \ge |P(\gamma, \tau)|. 
\] 
\end{lem}
\begin{proof}
The proof is essentially the same as in the proof of \cite[Theorem 4.1]{HuangLee}, except that we use our new results to deal with the boundary. Define $K_1$, $K_2$, and $\hat{\mathcal P}$ as in the proof of Proposition \ref{strictDEC}. More generally, for $(\gamma,\tau)\in\mathcal D$, we define
\begin{align*}
\hat{\mathcal P}_{(\gamma,\tau)}:[(1,0)+K_1]\times K_2&\to \mathcal L\times W^{1-\frac 1p,p}(\partial M)\\
((u,Y),(h,w))&\mapsto (\overline\Phi_{(g,\pi)},\Theta)[\Psi_{(\gamma,\tau)}(u,Y)+(h,w)],
\end{align*}
so that in particular, $\hat{\mathcal P}_{(g,\pi)}=\hat{\mathcal P}$. Using similar reasoning as in the proof of Lemma \ref{main},  we can use the inverse function theorem (together with estimates from Lemma \ref{AppendixA}) to see that there exists open ball $U\subset\mathcal{D}$ containing $(g, \pi)$ and constants $\delta>0$ and $C_1>0$ with the property that for all $(\gamma, \tau)$ in $U$, $\hat{\mathcal P}_{(\gamma, \tau)}$ is a diffeomorphism between a neighborhood of $((1,0),(0,0))$ in $[(1,0)+K_1]\times K_2$ and the ball of radius $\delta$ around 
$\hat{\mathcal P}_{(\gamma, \tau)}((1,0),(0,0))$ in $\mathcal L\times W^{1-\frac 1p,p}(\partial M)$, and 
\[
\| ((u-1, Y), (h,w))\|_{K_1\times K_2} \le C_1 \| \hat{\mathcal P}_{(\gamma, \tau)}((u, Y), (h, w)) - \hat{\mathcal P}_{(\gamma, \tau)}((1, 0), (0, 0)) \|_{\mathcal L\times W^{1-\frac 1p,p}}.
\]

We claim that the conclusion of Lemma~\ref{SobolevPMT} holds with this choice of $U$. We now assume $(\gamma, \tau)$ satisfies the hypotheses described in Lemma~\ref{SobolevPMT}, that is, $(\gamma, \tau)\in U$ such that $\overline{\Phi}_{(g,\pi)}(\gamma,\tau) = \overline{\Phi}_{(g,\pi)}(g, \pi)$, and $\Theta(\gamma, \tau)=\Theta(g,\pi)$. We want to show that $E(\gamma, \tau)\ge|P(\gamma,\tau)|$, and we will do this by constructing a sequence $(\bar\gamma_k, \bar\tau_k)$ that converges to $(\gamma,\tau)$, to which we can apply Theorem~\ref{PMT}).

 Select a sequence of smooth asymptotically flat initial data $(\gamma_k, \tau_k)$ converging to $(\gamma, \tau)$ in $\mathcal{D}$. This implies that $\hat{\mathcal P}_{(\gamma_k, \tau_k)}((1, 0), (0, 0))=\left(\overline{\Phi}_{(g,\pi)}(\gamma_k, \tau_k), \Theta(\gamma_k, \tau_k)\right)$ converges to $\hat{\mathcal P}_{(\gamma, \tau)}((1, 0), (0, 0))=\left(\overline{\Phi}_{(g,\pi)}(\gamma, \tau), \Theta(\gamma, \tau)\right)$ in $\mathcal L\times W^{1-\frac 1p,p}$.
 
 In particular, for large enough $k$, $(\gamma_k, \tau_k)\in U$ and 
 $\hat{\mathcal P}_{(\gamma, \tau)}((1,0),(0,0))$ lies in the $\delta$-ball around $\hat{\mathcal P}_{(\gamma_k, \tau_k)}((1,0),(0,0))$, and hence, by our construction of $U$,  there exists $((u_k, Y_k), (h_k, w_k))\in  [(1,0)+K_1]\times K_2$ such that 
\[
\hat{\mathcal P}_{(\gamma_k, \tau_k)}((u_k, Y_k), (h_k, w_k)) = 
\hat{\mathcal P}_{(\gamma, \tau)}((1, 0), (0, 0))
\]
and
\begin{align*}
\|((u_k-1, Y_k), (h_k,w_k))\|_{K_1\times K_2}
&\le C_1 \|  \hat{\mathcal P}_{(\gamma_k, \tau_k)}((u_k, Y_k), (h_k, w_k)) -\hat{\mathcal P}_{(\gamma_k, \tau_k)}((1, 0), (0, 0))
\|_{\mathcal L\times W^{1-\frac 1p,p}} \\
&= C_1 \|  \hat{\mathcal P}_{(\gamma, \tau)}((1, 0), (0, 0)) -\hat{\mathcal P}_{(\gamma_k, \tau_k)}((1, 0), (0, 0))
\|_{\mathcal L\times W^{1-\frac 1p,p}}.
\end{align*}
Setting $(\bar\gamma_k, \bar\tau_k):=\Psi_{(\gamma_k, \tau_k)}(u_k, Y_k) + (h_k, w_k)$, the inequality above shows that  $(\bar\gamma_k, \bar\tau_k)$ converges to $(\gamma, \tau)$ in $\mathcal{D}$. Note that 
\[ (\overline\Phi_{(g,\pi)}, \Theta)(\bar\gamma_k, \bar\tau_k)
= \hat{\mathcal P}_{(\gamma, \tau)}((1, 0), (0, 0))=  (\overline\Phi_{(g,\pi)}, \Theta)(\gamma, \tau)=  (\overline\Phi_{(g,\pi)}, \Theta)(g,\pi),
\]
and thus, unlike the arbitrary smoothing $(\gamma_k, \tau_k)$,  $(\bar\gamma_k, \bar\tau_k)$ satisfies the DEC (by Lemma~\ref{CH}) and has weakly outer trapped or inner untrapped boundary components. Moreover, by the same regularity argument used in the proof of Theorem~\ref{prescribed2}, $(\bar\gamma_k, \bar\tau_k)$ is smooth enough and decays enough so that the positive mass inequality (Theorem~\ref{PMT}) applies to $(\bar\gamma_k, \bar\tau_k)$, and hence $E(\bar\gamma_k, \bar\tau_k)\ge |P(\bar\gamma_k, \bar\tau_k)|$. Finally, we take the limit as $k\to\infty$ and use continuity of ADM energy-momentum~\cite[Lemma 8.4]{Lee} to conclude that $E(\gamma, \tau)\ge |P(\gamma,\tau)|$. (Note that $(\bar\gamma_k, \bar\tau_k)$ has the same modified constraints as $(\gamma, \tau)$, and thus the constraints of $(\bar\gamma_k, \bar\tau_k)$ converge to the constraints of $(\gamma, \tau)$ in $L^1$.)
\end{proof}

\subsection{Embedding in Minkowski space when $E=0$}\label{S4.2}

In this section, we use the Jang reduction method to show that if $\partial M\ne\emptyset$, then $E>0$. Combined with Theorem~\ref{PMT} and Theorem~\ref{ADMnull}, this implies Theorem~\ref{strictPMT}.  We recall that a function $f$ defined on an open set $U$ in an initial data set $(M,g,k)$ solves \emph{Jang's equation} \cite{Jang, SY81i} if
\begin{equation}\label{Jangs}H_g(f)-\tr_g(k)(f)=0,\end{equation}
where 
\[H_g(f)=\Div_g\left(\frac{\nabla f}{\sqrt{1+|\nabla f|^2}}\right)\]
is the mean curvature of the graph of $f$ in the cylinder over $(M,g)$ and 
\[\tr_g(k)(f)=\tr_g k-\frac{k(\nabla f,\nabla f)}{1+|\nabla f|^2}\]
is the trace of $k$ (extended trivially in the vertical direction) over the tangent spaces of the graph of $f$.

Jang's equation \eqref{Jangs} is a quasilinear elliptic equation for $f$, but the presence of the lower order term $\tr_gk$ precludes the use of the maximum principle to obtain a supremum estimate for $f$.\footnote{In the case when $\tr_g k$ has a good sign, see \cite[Theorem 3.4]{Metzger}.} The lack of such an a priori estimate is an obstacle for proving solutions of Jang's equation exist. Schoen and Yau \cite{SY81i} overcame this by instead considering the \emph{capillary regularized Jang's equation} 
\begin{equation}
    H_g(f_\tau)-\tr_g(k)(f_\tau)=\tau f_\tau \label{CapJang},
\end{equation}
where $\tau$ is a positive real parameter which we want to send to zero. 

The maximum principle now yields $\|f_\tau\|_{L^\infty}\les \tau^{-1}$, which is singular but suffices to show that the solutions $f_\tau$ exist globally. It follows that any global \emph{nonparametric} estimates (in the sense of minimal graphs) will grow like $\tau^{-1}$ as $\tau\to 0$. However, crucially, in the asymptotically flat setting $f_\tau(x)$ is bounded and even decays, uniformly in $\tau$, for $|x|$ sufficiently large in the asymptotically flat region \cite[Proposition 5]{E13}. To study the convergence of the $f_\tau$'s in the ``core," Schoen and Yau considered \emph{parametric} estimates, i.e.\ geometric estimates for the graphs of $f_\tau$. In fact, these graphs are \emph{$C$-minimizing} for some constant $C$ independent of $\tau$ (for this definition we refer to \cite{DS93,E09}). For such hypersurfaces the compactness and regularity theory is essentially the same as for area minimizing hypersurfaces. It follows that the \emph{graphs} of $f_\tau$ converge smoothly as hypersurfaces in $M\times \Bbb R$ as $\tau\to 0$. The components of the limit are either graphs of solutions to Jang's equation \eqref{Jangs} or cylinders over MOTS or MITS in the data set $(M,g,k)$.\footnote{The graphical components tend to $\pm\infty$ on approach to these cylinders. We say that the Jang graph ``blows up" over the MOTS or MITS.} Since $|f_\tau(x)|\les 1$ for $|x|\gtrsim 1$, the limiting hypersurface contains a graphical component, defined over some set $\mathcal U$ which must contain a neighborhood of infinity. It is precisely this exterior graphical component that is studied in the works \cite{SY81i,E13}. These basic properties of Jang's equation are summarized neatly in \cite[Proposition 7]{E13}. 

We can now describe our modification of Eichmair's argument in the boundary case. 

\begin{thm}\label{E=0}
Let $3\le n\le 7$, and let $(M^n,g,k)$  be a complete asymptotically flat initial data set with \emph{nonempty} compact boundary $\partial M$ such that the dominant energy condition holds on $M$ and each component of $\partial M$ is either weakly outer trapped or weakly inner untrapped. In the case $n=3$, we also assume that  $\tr_g k=O(|x|^{-\gamma})$ for some $\gamma>2$. Then $E>0$. 
\end{thm}
\begin{proof}
Let $(M, g, k)$ be as described in the hypotheses.
By our DEC density theorem, Theorem \ref{DECdensity2}, there exists a sequence of initial data $(g_j,k_j)\to (g,k)$ on $M$ with harmonic asymptotics, satisfying the strict dominant energy condition, $\theta^+_j<0$ on $\partial^+ M$, $\theta^-_j>0$ on $\partial^-M$, and $|\tr_{g_j}k_j|\le C|x|^{-\gamma}$ uniformly in $j$ when $n=3$.\footnote{The $n=3$ claim follows from observing that the only term appearing in $\tr_{g_j}k_j$ (when written in terms of $g$, $k$, $\lambda$, and the deformations $u, Y h, w$) that is not directly controlled is the $\tr_g k$ term, which is controlled by assumption. See \cite[Proposition 15]{E13}.}


The strict sign for $\theta^\pm$ on $\partial M$ relative to $(g_j,k_j)$ allows $\partial M$ to act as a barrier for the capillary regularized Jang equation 
\[H_{g_j}(f_\tau)-\tr_{g_j}(k_j)(f_\tau)=\tau f_\tau\]
as in \cite[Proposition 3.5]{AM09}. On the asymptotically flat end of $M$, we prescribe $f_\tau\to 0$ and proceed to solve as in \cite{SY81i, AM09, E13}. We also obtain the usual parametric and nonparametric estimates associated to Jang's equation. Letting $\tau \to 0$, we obtain open sets $\mathcal U_j\subset M$, as described above, containing $\{|x|\ge R_0\}$ for some large $R_0$, equipped with a function $f_j\in C^{3,\alpha}_{-n+2+\eta}(\mathcal U_j)$ for any fixed $\eta>0$ which solves Jang's equation and satisfies the properties proved in \cite[Proposition 7]{E13}, with the same proof. The barrier property of $\partial M$ implies that $f_j$ is unbounded. In particular, $f_j$ must blow up over some nonempty union of closed MOTS and MITS enclosing $\partial M$. More specifically, the graph of $f_j$ must have at least one end that is asymptotic to a cylinder~\cite[Proposition 7 (c)]{E13}.

\begin{claim}\label{JangClaim}
After passing to a subsequence, the graphs of $f_j$ converge in $C^{3,\alpha}_\loc$ to the graph of a Jang solution $f:\mathcal U\to \Bbb R$, which blows up over some nonempty union of closed MOTS and MITS  in $(M,g,k)$. With its induced metric, the graph  of $f$ is asymptotically flat with a nonzero number of ends that are asymptotically cylindrical. 
\end{claim}

The main nontrivial claim here is that the property of having an asymptotically cylindrical end persists in the limit.  By the Harnack inequality for Jang's equation, it suffices to show that the limiting function $f$ is unbounded. The only thing we must rule out is cylindrical ends collapsing into the boundary $\partial M$. To do this, we extend the manifold $M$ to a slightly larger manifold $\tilde M$ so that $\partial M$ lies in the interior of $M$. We extend each metric $g_j$ (including $g$) to $\tilde M$ so that on every compact set, $\|g_j-g\|_{C^{2,\alpha}}\to 0$. Then we apply the compactness and regularity theory for $C$-minimizing graphs on the extended manifold. The Jang graphs no longer approach the boundary, so the cylindrical ends cannot disappear in the limit. The claims about the blow-up locus being a collection of MOTS/MITS in $(M,g,k)$ and the graph being asymptotically flat follow easily from \cite[Proposition 7]{E13}. This proves Claim \ref{JangClaim}. 

We assume now that $E=0$ and work to obtain a contradiction.
It follows that  $E_j\to 0$, and arguing as in \cite[Proposition 16]{E13}, we see that the graph of $f$, which we denote by $\Sigma$, is scalar-flat and has zero mass. (This part of the argument is highly nontrivial but is agnostic to the presence of the boundary $\partial M$; it only relies on the conclusion of Claim \ref{JangClaim}.) Now $(\Sigma,g_\Sigma)$, viewed as a time-symmetric initial data set, must be diffeomorphic to $\Bbb R^n$. This can be viewed as an extension of the rigidity of the Riemannian positive mass theorem to manifolds with cylindrical ends. The argument is outlined in \cite{E13}, but is a special case of the more recent positive mass theorem with \emph{arbitrary ends} \cite{LUY21, LLU2}. One first shows that $\Sigma$ is Ricci-flat and then has only one end by the Cheeger--Gromoll splitting theorem. However, $\Sigma$ having only one end is in contradiction to Claim \ref{JangClaim}. 
\end{proof}

\begin{appendix}

\section{Second differential of the constraint-null expansion system}\label{appendix}

In this paper, we utilize the inverse function theorem to perturb \emph{families} of initial data sets. To this end, we need to control the constants appearing in the ``quantitative" version of the inverse function theorem \cite[Theorem A.43]{Lee}. 

\begin{lem}\label{AppendixA} Let $(M^n,g, \pi)$ be an asymptotically flat data set as in Section \ref{defs}. Let $K_1\subset T_{(1,0)}\mathcal C$ be a closed subspace and $K_2\subset T_{(g,\pi)}\mathcal D$ be a finite-dimensional subspace. There exists a constant $C_0$ such that for any $r_0>0$ sufficiently small, the following is true.

Let $(\gamma,\tau)\in \mathcal D$ with $\|(\gamma,\tau)-(g,\pi)\|_\mathcal D\le r_0$ and define
\begin{align*}\hat{\mathcal P}_{(\gamma,\tau)}:[(1,0)+K_1]\times K_2&\to \mathcal L\times W^{1-\frac 1p,p}(\partial M)\\
((u,Y),(h,w))&\mapsto (\Phi,\Theta)[\Psi_{(\gamma,\tau)}(u,Y)+(h,w)].\end{align*}
Then
\begin{equation}\label{first}\|D\hat{\mathcal P}_{(\gamma,\tau)}|_{((1,0),(0,0))}-D\hat{\mathcal P}_{(g,\pi)}|_{((1,0),(0,0))}\|_{L(K_1\times K_2,\mathcal L\times W^{1-\frac 1p,p})}\le C_0r_0\end{equation}
and
\begin{equation}\label{second}\|D^2 \hat{\mathcal P}_{(\gamma,\tau)}|_{((u,Y),(h,w))}\|_{L_2(K_1\times K_2,\mathcal L\times W^{1-\frac 1p,p})}\le C_0 \end{equation}
for any $((u,Y),(h,w))\in B_{r_0}((1,0),(0,0))$.

This lemma also holds if in the definition of $\hat{\mathcal P}$, we use the modified constraint operator $\overline{\Phi}_{(g,\pi)}$ instead of $\Phi$. 
\end{lem}

Here $L_2(X,Y)$ refers to the space of bounded multilinear maps $X\times X\to Y$. Note that a Lipschitz bound for $D\hat{\mathcal P}_{(\gamma,\tau)}$ follows from the Hessian bound by the mean value theorem in Banach spaces. The proof proceeds with a computation of $D\Phi,D^2\Phi,D\Theta$, and $D^2\Theta$. 

\begin{lem} \label{PhiSecondDerivative}
The first derivative (linearization) of the constraint operator is given by 
\begin{align}\label{DPhi}D\Phi|_{(g,\pi)}(h,w)&=\bigg(-\Delta_g(\tr_g h)+\Div_g(\Div_g h)-\langle \Ric_g,h\rangle_g+\tfrac{2}{n-1}(\tr_g\pi)\left(\pi^{ij}h_{ij}+\tr_g w\right)\\
&\quad\quad - 2g_{kl}\pi^{ik}\pi^{jl}h_{ij}-2\langle \pi,w\rangle_g,\nonumber\\
&\quad\quad(\Div_g w)^i-\tfrac 12 g^{ij}\pi^{kl}\nabla_j h_{kl}+g^{ij}\pi^{kl}\nabla_ kh_{jl}+\tfrac 12 \pi^{ij}\nabla_j (\tr_g h)\bigg)\nonumber
\end{align}
Schematically, the second derivative is given by 
\begin{align}\label{D2Phi}D^2\Phi|_{(g,\pi)}((h_1,w_1),(h_2,w_2))&=\bigg(\sum_{0\le i_1+i_2\le 2}\nabla^{i_1}h_1*\nabla^{i_2}h_2+\Riem * h_1 * h_2\\
&\quad\quad +\pi*\pi*h_1*h_2+w_1*w_2+\pi*h_1*w_2+\pi*h_2*w_1,\nonumber\\
&\quad\quad  w_1*\nabla h_2+w_2*\nabla h_1+ \pi* h_1*\nabla h_2+\pi* \nabla h_1*h_2\bigg). \nonumber
\end{align}
Here we use the usual schematic notation where $A*B$ denotes linear combinations and contractions of the components of $A$ and $B$ with respect to the metric $g$. 
\end{lem}

The schematic notation misses factors of $g$ and $g^{-1}$ but these are pointwise bounded by Morrey's inequality. In the following calculation, we use the shorthand $\delta_g F=DF|_{g}(h)$. 

\begin{proof}
The formula for $D\Phi$ is well known in the literature \cite{FischerMarsden}. It depends on the linearization of the scalar curvature, which can be found in \cite{Lee}, for instance. To obtain the formula for $D^2\Phi$, we simply differentiate \eqref{DPhi}, making note of the following rules: 
\begin{itemize}
    \item $\delta_g \nabla T=\nabla h_2*T + \nabla \delta_g T$ for any tensor $T$, and 
    
    \item contractions produce terms of $h_2$ $*$ what was being contracted. 
\end{itemize}
Finally, we also note that the variation of the Ricci tensor is given by 
\[-2\delta_g R_{ij}=\Delta_Lh_{2\,ij}+\nabla_i\nabla_j \tr_g h_2-\nabla_i(\Div_g h_2)_j-\nabla_j(\Div_g h_2)_i,\]
where $\Delta_L$ is the Lichnerowicz Laplacian. In our schematic notation, this becomes 
\[\delta_g \Ric=\nabla^2h_2+\Riem * h_2.\]
The variation in $\pi$ is much more straightforward and \eqref{D2Phi} is easily obtained along these lines. 
\end{proof}

\begin{lem}\label{ThetaSecondDerivative}
The first derivative of the boundary null expansion is given by 
\begin{equation}\label{DTheta}D\Theta|_{(g,\pi)}(h,w)=\tfrac 12 \tr_{\partial M}(\nabla_\nu h) -\Div_{\partial M} \omega - \tfrac 12 h(\nu,\nu)H-h(\nu,\nu)\pi(\nu^\flat,\nu^\flat)-w(\nu^\flat,\nu^\flat),\end{equation}
where $\omega_i=h_{ij}\nu^j-h(\nu,\nu)\nu_i$ and $\nu^\flat$ denotes the $1$-form dual to $\nu$. Schematically, the second derivative is given by 
\begin{equation}\label{D2Theta}D^2\Theta|_{(g,\pi)}((h_1,w_1),(h_2,w_2))=\sum_{0\le i_1+i_2\le 1}\nabla^{i_1}h_1*\nabla^{i_2}h_2+h_1*w_2+h_2*w_1+w_1*w_2,\end{equation}
where schematic notation here is omitting terms like $\nu$ and $H$. 
\end{lem}
\begin{proof}
We first compute the linearization of the normal. Varying $g(\nu,\nu)=1$ gives \[h(\nu,\nu)+2g(\nu,\delta_g\nu)=0,\]
while varying $g(X,\nu)=0$ for $X\in T\Sigma$ gives
\[h(X,\nu)+g(X,\delta_g\nu)=0.\]
It follows that 
\begin{equation}
    \delta_g \nu^i=-h^{ij}\nu_j+\tfrac 12 h(\nu,\nu)\nu^i.  
\end{equation}

Secondly, we compute the linearization of the second fundamental form. For $X$ and $Y$ tangent to $\Sigma$, we have 
\[A(X,Y)=-g(\nu,\nabla_X Y)=-g_{ij} \nu^i X^k \nabla_k Y^j.\] 
Taking the variation, we have
\begin{align*}
    \delta_g A(X,Y)&=-h_{ij}\nu^i X^k\nabla_k Y^j-g_{ij}\left(-h^{il}\nu_l+\tfrac 12 h(\nu,\nu)\nu^i\right)X^k\nabla_k Y^j-g_{ij}\nu^i X^k\delta_g(\nabla_kY^j)\\
    &= \tfrac 12 h(\nu,\nu)\left(-g_{ij}\nu^i X^k\nabla_k Y^j\right)-g_{ij}\nu^i X^k\delta_g(\nabla_kY^j)\\
    &=\tfrac 12 h(\nu,\nu)A(X,Y)-\tfrac 12 g_{ij}\nu^i g^{jm} \left(\nabla_k h_{lm} +\nabla_l h_{km}-\nabla_m h_{kl}\right) X^k Y^l\\
    &=\tfrac 12\left(h(\nu,\nu)A_{kl}-\nu^m\nabla_k h_{lm}-\nu^m\nabla_l h_{km}+\nu^m\nabla_m h_{kl}\right)X^kY^l.
\end{align*}
Now
\[\nu^m\nabla_k h_{lm}=\nabla_k \omega_l +h(\nu,\nu)A_{kl} - h_{ln}A_k{}^n,\]
so that finally
\begin{equation}
    \label{Avariation} \delta_g A(X,Y)= \tfrac{1}{2}(\nabla_\nu h_{ij}-\nabla^{\partial M}_i\omega_j -\nabla_j^{\partial M}\omega_i+h_{ik}A_j{}^k +h_{jk}A_i{}^k - h(\nu,\nu)A_{ij})X^iY^j.
\end{equation}
The mean curvature of the boundary is given by 
\[H=\tr_{\partial M}A=(g^{ij}-\nu^i\nu^j)A_{ij},\]
so taking the variation and using \eqref{Avariation} yields
\begin{equation}\delta_g H=\tfrac 12 \tr_{\partial M}(\nabla_\nu h) -\Div_{\partial M} \omega - \tfrac 12 h(\nu,\nu)H\label{Hvariation}.\end{equation}
The formula for $D\Theta$ follows easily, where also note that 
\[\delta_g\nu^\flat = \tfrac 12 h(\nu,\nu)\nu^\flat.\] The schematic computation for $D^2\Theta$ also follows easily using the rules establised in the proof of Lemma \ref{PhiSecondDerivative}. 
\end{proof}

From these formulas, we deduce:

\begin{lem}
There exists a constant $C_0$ such that for any sufficiently small $r_0>0$ the following is true. If $(\gamma,\tau)\in \mathcal D$ satisfies $\|(\gamma,\tau)-(g,\pi)\|_\mathcal D\le r_0$, then 
\begin{equation}
    \|D(\Phi,\Theta)|_{(\gamma,\tau)}-D(\Phi,\Theta)|_{(g,\pi)}\|\le C_0r_0 \label{grad1}
\end{equation}
and 
\begin{equation}
     \|D^2(\Phi,\Theta)|_{(g,\pi)}\|\le C_0.\label{grad2}
\end{equation}
\end{lem}
\begin{proof} We first remark that the constants appearing in the Sobolev, Morrey, and trace inequalities associated to the metric $\gamma$ can be bounded in terms of $r_0$.
The first estimate \eqref{grad1} can be read off from the explicit formulas \eqref{DPhi} and \eqref{DTheta}. For example, consider 
\[g^{ij}\partial_i\partial_j (g^{kl}h_{kl})-\gamma^{ij}\partial_i\partial_j(\gamma^{kl}h_{kl}).\]
We rewrite this as 
\[(g^{ij}-\gamma^{ij})\partial_i\partial_j(\gamma^{kl}h_{kl})+g^{ij}\partial_i\partial_j((g^{kl}-\gamma^{kl})h_{kl})\]
and from this it is not hard to see that the $L^p_{-q}$ norm can be estimated by $\les r_0 \|h\|_{W^{2,p}_{-q}}$. 

To prove the estimate \eqref{grad2}, we examine the bilinear structure of the schematic formulas \eqref{D2Phi} and \eqref{D2Theta}. For $D^2\Phi$, we put the highest number of derivatives in $L^p_{-q}$ and the lowest number of derivatives in $L^\infty$ using Morrey's inequality. Special care must be taken with the $\Riem *h_1*h_2$ term, as the curvature is not assumed to be pointwise bounded. However, it is in $L^p_{-q}$, so we just put $h_1$ and $h_2$ in $L^\infty$. Altogether, we obtain the estimate 
\[\|D^2\Phi|_{(\gamma,\tau)}((h_1,w_1),(h_2,w_2))\|_\mathcal L\les \|(h_1,w_1)\|_{W^{2,p}_{-q}\times W^{1,p}_{-q-1}}\|(h_2,w_2)\|_{W^{2,p}_{-q}\times W^{1,p}_{-q-1}}.\]
For $D^2\Theta$, we estimate each of the terms appearing in \eqref{D2Theta} in $W^{1-\frac 1p,p}(\partial\Omega)$. Terms with derivatives are handled using Lemma \ref{trace} instead of Morrey's inequality. Note that our schematic notation omits the normal $\nu_g$ and mean curvature $H_g$, however both of these are pointwise bounded in terms of $\gamma$. Therefore, we obtain the estimate 
\[\|D^2\Theta|_{(\gamma,\tau)}\|_{W^{1-\frac 1p,p}}\les\|(h_1,w_1)\|_{W^{2,p}_{-q}\times W^{1,p}_{-q-1}}\|(h_2,w_2)\|_{W^{2,p}_{-q}\times W^{1,p}_{-q-1}}, \]
as desired. 
\end{proof}

We can now prove the main result of this appendix, Lemma \ref{AppendixA}. 

\begin{proof}[Proof of Lemma \ref{AppendixA}]
We first define a function 
\begin{align*}
    \hat\Psi_{(\gamma,\tau)}:[(1,0)+K_1]\times K_2&\to \mathcal D\\
   ((u,Y),(h,w)) &\mapsto \Psi_{(\gamma,\tau)}(u,Y)+(h,w),
\end{align*}
so that
\[\hat{\mathcal P}_{(\gamma,\tau)}=(\Phi,\Theta)\circ \hat{\Psi}_{(\gamma,\tau)}.\]
By the chain rule for functions on Banach spaces, \begin{equation}\label{a1}D\hat{\mathcal P}_{(\gamma,\tau)}((v_1,Z_1),(h_1,w_1))=D(\Phi,\Theta)\circ D\hat\Psi_{(\gamma,\tau)}((v_1,Z_1),(h_1,w_1)).\end{equation}
The second derivative is given by 
\begin{multline}\label{a2}
    D^2\hat{\mathcal P}_{(\gamma,\tau)}\Big(((v_1,Z_1),(h_1,w_1)),((v_2,Z_2),(h_2,w_2))\Big)=\\
    D^2(\Theta,\Phi)\Big(D\hat\Psi_{(\gamma,\tau)}((v_1,Z_1),(h_1,w_1)),D\hat\Psi_{(\gamma,\tau)}((v_2,Z_2),(h_2,w_2))\Big)\\
    +D(\Theta,\Phi)\circ D^2\hat\Psi_{(\gamma,\tau)}\Big(((v_1,Z_1),(h_1,w_1)),((v_2,Z_2),(h_2,w_2))\Big).
\end{multline}
The derivatives of $\hat{\Psi}_{(\gamma,\tau)}$ are given by 
\begin{equation}\label{a3}
    D\hat{\Psi}_{(\gamma,\tau)}((v_1,Z_1),(h_1,w_1))=(su^{s-1}v_1\gamma+h_1, -\tfrac 32 s u^{-\frac 32 s-1}v(\tau+\mathfrak L_\gamma Y)+ u^{-\frac 32s}\mathfrak L_\gamma Z_1+w_1)
\end{equation}
and 
\begin{multline}\label{a4}
    D^2\hat{\Psi}_{(\gamma,\tau)}\Big(((v_1,Z_1),(h_1,w_1)),((v_2,Z_2),(h_2,w_2))\Big)=(s(s-1)u^{s-2}v_1v_2,\\
  \tfrac 32 s(\tfrac 32 s+1)u^{-\frac 32 s-2}v_1v_2(\tau+\mathfrak L_\gamma Y)-\tfrac 32 s u^{-\frac 32 s-1}v_2\mathfrak L_\gamma Z_1 -\tfrac 32 s u^{-\frac 32 s-1}v_1\mathfrak L_\gamma Z_2  ).
\end{multline}
In these formulas, the differentials are being evaluated at $((u,Y),(h,w))$ or $\hat\Psi_{(\gamma,\tau)}((u,Y),(h,w))$, wherever appropriate.

To prove \eqref{first}, we use \eqref{a1} for $(\gamma,\tau)$ and $(g,\pi)$ at $((1,0),(0,0))$, which yields 
\[D\hat{\mathcal P}_{(\gamma,\tau)}-D\hat{\mathcal P}_{(g,\pi)}=D(\Phi,\Theta)|_{(\gamma,\tau)}(D\hat\Psi_{(\gamma,\tau)}-D\hat\Psi_{(g,\pi)})+(D(\Phi,\Theta)|_{(\gamma,\tau)}-D(\Phi,\Theta)|_{(g,\pi)})D\hat\Psi_{(g,\pi)}.\]
For $(\gamma,\tau)$ sufficiently close to $(g,\pi)$, we may evidently estimate both of these terms (in operator norm) using \eqref{a3} and the estimate \eqref{grad1}. 

To prove \eqref{second}, we note that \eqref{a2} implies 
\begin{align*}\|D^2\hat{\mathcal P}_{(\gamma,\tau)}|_{((u,Y),(h,w))}\|&\le  
\left\|D^2(\Theta,\Phi)|_{\hat\Psi_{(\gamma,\tau)}((u,Y),(h,w))}\right\|\cdot \left\|D\hat\Psi_{(\gamma,\tau)}|_{((u,Y),(h,w))}\right\|^2\\
&\quad  \quad + \left\|D(\Theta,\Phi)|_{\hat\Psi_{(\gamma,\tau)}((u,Y),(h,w))}\right\|\cdot\left\|D^2\hat\Psi_{(\gamma,\tau)}|_{((u,Y),(h,w))}\right\|.\end{align*}
For $((u,Y),(h,w))$ small, $\hat{\Psi}_{(\gamma,\tau)}((u,Y),(h,w))$ is close to $(g,\pi)$ in $\mathcal D$, so we may apply \eqref{grad1} and \eqref{grad2}. Furthermore, the same estimates may be derived for $D\hat\Psi$ and $D^2\hat\Psi$ from \eqref{a3} and \eqref{a4}. This completes the proof of \eqref{second}.\end{proof}

\end{appendix}

\bibliographystyle{alpha}
\bibliography{bib}
\end{document}